\documentclass[10pt]{amsart}
\usepackage{amssymb,latexsym,pstricks}
\usepackage[normalem]{ulem}
\topskip  \usepackage{color}
\usepackage{orcidlink}
\usepackage{hyperref}
\usepackage{tensor}
\usepackage{enumitem}
\usepackage{mathabx}
\usepackage{comment}
\usepackage{amsfonts,amssymb,amsmath,amsthm,graphicx,latexsym}
\usepackage{mathrsfs,dsfont}
\usepackage{pstricks,pst-node,pst-plot}
\usepackage{pgfplots}
\usepackage{tikz}
\usepackage{bm}
\usepackage{caption}
\theoremstyle{thmstyleone}%
\newtheorem{theorem}{Theorem}[section]
\newtheorem{proposition}[theorem]{Proposition}%
\newtheorem{lemma}[theorem]{Lemma}%
\newtheorem{corollary}[theorem]{Corollary}

\theoremstyle{thmstyletwo}%
\newtheorem{example}[theorem]{Example}%
\newtheorem{remark}[theorem]{Remark}%

\theoremstyle{thmstylethree}%
\newtheorem{definition}[theorem]{Definition}%

\newcommand{\tp}{\mathbin{\hbox{$\bigcirc$\hbox to 0pt{\hspace{-0.81em}$\scriptstyle\top$\hfil}}}}

\begin{document}
\title[Normal ordering in the $(p,q)$-deformed generalized Weyl algebra]{Normal ordering in the $(p,q)$-deformed generalized Weyl algebra. I: Algebraic Framework and Combinatorial Identities}
\author[T. Mansour]{Toufik Mansour $^{1, \orcidlink{0000-0001-8028-2391}}$}
\address{$^1$ Department of Mathematics, University of Haifa, 3498838 Haifa, Israel, tmansour@univ.haifa.ac.il}
\author[L. Oussi]{Lahcen Oussi $^{2, \orcidlink{0000-0001-9804-0761}}$}
\address{$^2$ Faculty of Pure and Applied Mathematics, Wrocław University of Science and Technology, Wybrzeże Wyspiańskiego 27, 50-370 Wrocław, Poland, lahcen.oussi@pwr.edu.pl, oussimaths@gmail.com}
\author[M. Schork]{Matthias Schork $^{3, \orcidlink{0000-0002-9759-6618}}$}
\address{$^3$ Institute for Mathematics, Würzburg University, Emil-Fischer Str. 40, 97074 Würzburg, Germany, matthias.schork@mathematik.uni-wuerzburg.de}


\begin{abstract}
The $(p,q)$-deformed generalized Weyl algebra is generated by variables $X, Y$ and $Z_p$ which satisfy the commutation relations $XY-qYX=h Y^sZ_{p}, XZ_p=pZ_pX$, and $Z_pY=pYZ_p$, with $s\in \mathbb{N}_0$ and where $p^s=1$. We investigate the problem of normal ordering arbitrary words in these letters with the help of Young diagrams, and we treat certain special cases explicitly. In particular, the connection to generalized Stirling numbers is considered in detail.
\end{abstract}
\maketitle
\section{Introduction}\label{sec1}
A commutation relation characterizes the disparity between different orders of operation of given operators $X$ and $Y$. For notation, we use the commutator $[X, Y]:=XY-YX$ to describe this relation. Note that, if $X$ and $Y$ commute (that is, $XY=YX$), then the commutator vanishes. There are many examples for noncommutative settings. In particular, matrices, quaternions, Lie algebras, etc., provide many instances. Noncommutative structures play a central role in discrete mathematics and physics (e.g., harmonic oscillator in quantum mechanics). The {\em Weyl algebra} can be considered as the complex unital algebra generated by two operators $X$ and $Y$ satisfying the commutation relation 
\begin{equation}\label{Walg}
[X, Y]=I, 
\end{equation}
where $I$ denotes the identity operator. A concrete representation of the generators is given as operators acting on (smooth) functions of a single variable: $X \rightsquigarrow D=\frac{d}{dx}$ is the derivative operator with respect to $x$, and $Y\rightsquigarrow M$ is the multiplication operator with $x$ (i.e., $(Mf)(x)=xf(x)$). Thus, from a mathematical point of view, the Weyl algebra is the simplest example of a ring of differential operators and has been studied thoroughly due to its rich structure. From a physical point of view, a realization is given by $X\rightsquigarrow \hat{a}$ and $Y \rightsquigarrow \hat{a}^{\dag}$, i.e., the  annihilation and creation operator of the harmonic oscillator. A word in the letters $X$ and $Y$ is in {\em normal ordered form} if all letters $Y$ stand to the left of all the letters $X$. It is an old result due to Scherk from 1823 (see \cite{TMMS2016} for some history) that normal ordering $(YX)^n$ (in the concrete representation mentioned above) involves the Stirling numbers of the second kind,
\begin{equation}\label{Scherk}
(YX)^n=\sum_{k=1}^n S(n,k) Y^kX^k.
\end{equation}
In the physics community, this result was rediscovered by Katriel in 1974 (in the realization with the creation and annihilation operators). Many generalizations have been considered, see, e.g.,  \cite{BF2011,TMMS2016}. \

In the past years, many ``$q$-deformed'' structures have been studied intensively by many authors. By a {\em $q$-deformation} we mean introducing a parameter $q$ into defining relations such that by taking the limit $q\rightarrow 1$ the undeformed structure is recovered. For instance, the {\em $q$-deformed Weyl algebra} is defined by the commutation relation $XY-qYX=I$, and it is clear that, for $q\rightarrow 1$, we obtain \eqref{Walg}. The {\em quantum plane} \cite{YIM1988} is the algebra generated by two variables $X$ and $Y$ satisfying the $q$-commutation relation
\begin{equation}\label{qcrintro}
XY=qYX.
\end{equation}
The consequences of \eqref{qcrintro} have been studied in the mathematical literature and there exists a deep connection to identities for $q$-special functions, see, e.g., \cite{KHT1995, KTH1997}. In the physical literature, many $q$-deformed models have been investigated. As in the undeformed case, the commutation relation of the $q$-deformed Weyl algebra is realized by the $q$-deformed annihilation and creation operators. Furthermore, one has a $q$-analog to \eqref{Scherk} involving the $q$-deformed Stirling numbers of the second kind. In more recent times, some of the $q$-deformed structures have also been {\em $(p,q)$-deformed}, i.e., a further deformation parameter $p$ has been introduced into the defining relations such that, for $p\rightarrow 1$, the $q$-deformed structure is recovered. There exist mathematical as well as physical motivations to study such $(p,q)$-deformations. For example, the $q$-deformed Weyl algebra has been further deformed from a mathematical motivation in \cite{JG2016,EKLS1993}, and the commutation relation of the $q$-deformed annihilation and creation operator has been extended to a $(p,q)$-deformation, leading in particular to $(p,q)$-deformed Stirling numbers upon normal ordering \cite{KK1992} (see also \cite{AB2010,AB2015,HOT2026,OT2024,OT2026} for some related recent works) and to $(p,q)$-special functions in analogy to the $q$-special functions mentioned above \cite{BK1994,Jag1997}. The $(p,q)$-deformed Stirling numbers were also discussed around this time in connection with a $(p,q)$-statistics for rook placements on the staircase board \cite{MWDW1991}. Since then, many further aspects and generalizations have been studied. For more background and some history concerning normal ordering, see \cite{BF2011,TMMS2016,MS2021} and the references therein.

Let us describe the contents of this paper. We consider variables $X, Y$ and $Z_p$ satisfying the following $(p,q)$-commutation relations, 
\begin{align}
   XY-qYX &= hf(Y)Z_p,\label{pqcr1} \\
  Z_pY&=pYZ_p, \label{pqcr2}\\
  XZ_p&=pZ_pX, \label{pqcr3}\\
  Z_1 &=I,\label{pqcr4}
\end{align}

where $f$ is a polynomial function. For a word $\omega$ in letters $X,Y$ and $Z_p$, we say it is in {\em normal ordered form} if all the letters $Y$ stand to the left of all the letters $Z_p$ which stand to the left of all the letters $X$. That is, $\omega$ has the form $\omega=\sum_{k,l,m} a_{k,l,m}Y^kZ_p^lX^m$ for some coefficients $a_{k,l,m}$. For the special case where $f(Y)=Y^s$ is a monomial -- where we call the algebraic structure generated by  $X,Y$ and $Z_p$ the {\em $(p,q)$-deformed generalized Weyl algebra} -- we derive the normal ordered form of arbitrary words and give a combinatorial interpretation of the normal ordering coefficients in terms of generalized Stirling numbers. We discuss many special cases by specializing the parameters $s,h,p,q$, thereby recovering several results given in literature as well as deriving new ones. For example, for $p=1$, we obtain the $q$-commutation relation $XY-qYX = hf(Y)$ studied in \cite{CCG2017,TMMS2011}. Letting furthermore $q=1$, the resulting algebra was studied in depth in \cite{BLO2013,BLO2015,BLO2015a}. On the other hand, specializing to $X=D_{p, q}$, the $(p, q)$-derivative operator, $Y$ the operator of multiplication with $x$, and $Z_p$ the Fibonacci operator \cite{GCMDCH}, one gets the $(p, q)$-commutation relations investigated in \cite{LO2021}. 

This paper is the first part of a series of 3 papers. In a sequel \cite{TLM2}, we introduce $(p,q)$-rook numbers and give an interpretation of the normal ordering coefficients of words in the $(p,q)$-deformed generalized Weyl algebra in terms of these $(p,q)$-rook numbers. In further work \cite{TLM3}, the binomial formula for $(X+Y)^n$ is studied in the $(p,q)$-deformed generalized Weyl algebra.

The structure of the present paper is as follows. In Section~\ref{mr}, we fix some notation and present basic results on normal ordering and point out special cases which correspond to well-known structures. In Section~\ref{gno}, we extend the normal ordering given in Section~\ref{mr} to arbitrary words when $f(Y)=Y^s$ -- i.e., in the $(p,q)$-deformed generalized Weyl algebra -- and discuss the connection to (generalized) Stirling numbers. Finally, in Section~\ref{concl}, we present some conclusions.
 
\section{Basic results on normal ordering}\label{mr}
In this section, we present basic normal ordering results for $X,Y$ and $Z_p$ satisfying  \eqref{pqcr1} - \eqref{pqcr4}. Before we do this, we fix some notations that will be used throughout the paper. Let $p$ and $q$ be two indeterminates. We define the {\em twin-basic numbers} for $n\in \mathbb{N}_0:=\mathbb{N}\cup \{0\}$ by
$$
[n]_{p, q}=\frac{p^n-q^n}{p-q}.
$$ 
We regard them as a $(p,q)$-deformation of the natural numbers. Clearly, specializing to $p=1$, one recovers the conventional $q$-deformed numbers, $[n]_{1, q}=\frac{1-q^n}{1-q}=[n]_{q}$. One has $[n]_{p, q}=p^{n-1}+p^{n-2}q+\cdots + pq^{n-2}+q^{n-1}$, reducing for $p=1$ to the well-known identity $[n]_{q}=1+q+q^2+\cdots + q^{n-1}$. The corresponding factorials and binomial coefficients are given by
$$
[n]_{p, q}!=[n]_{p, q}[n-1]_{p, q}\cdots [2]_{p, q}[1]_{p, q},\,\,\, {n\choose k}_{p, q}=\frac{[n]_{p, q}!}{[n-k]_{p, q}![k]_{p, q}!},
$$
together with $[0]_{p, q}!=1$. One may introduce a {\em $(p,q)$-derivative} $D_{p,q}$ by letting it act as
$$
(D_{p,q}f)(x):=\frac{f(px)-f(qx)}{(p-q)x}.
$$
For $p=1$, one recovers the {\em Jackson-derivative} $D_q$, i.e., $(D_{1,q}f)(x)=\frac{f(x)-f(qx)}{(1-q)x}=(D_{q}f)(x)$.  The action of $D_{p,q}$ on monomials is given, for $n \in \mathbb{N}$, by 
$$
D_{p,q}x^n=[n]_{p, q}x^{n-1}.
$$
See \cite{LO2024} for more information on $D_{p,q}$ and some references.

\subsection{The normal ordered form of $X^{m}Y$}
Let $X, Y$ and $Z_p$ be variables satisfying \eqref{pqcr1} - \eqref{pqcr4}, such that $f$ is a polynomial function in \eqref{pqcr1}. In this section, we present an identity for normal ordering $X^mY$. For instance, it is clear that from \eqref{pqcr1}, one obtains
\begin{equation}\label{p2no}
X^2Y=q^2YX^2+qhf(Y)Z_pX+hXf(Y)Z_p.
\end{equation}
Hence, we have to consider the normal ordered form of the last term on the right-hand side in \eqref{p2no}. For that, we need the following result.
\begin{lemma}\label{lem1}
Let $X, Y$ and $Z_p$ be variables satisfying \eqref{pqcr1} - \eqref{pqcr4}. For $k \in \mathbb{N}$, we have
\begin{equation}
XY^k=q^kY^kX+h[k]_{p, q}f(Y)Y^{k-1}Z_p\label{eq1}.
\end{equation}
\end{lemma}
\begin{proof}
The proof will be by induction on $k$. For $k=1$, \eqref{eq1} holds true by assumption. By the induction hypothesis, we have
\begin{align*}
XY^{k+1}&=q^kY^kXY+h[k]_{p, q}f(Y)Y^{k-1}Z_{p}Y\\
&=q^kY^k(qYX+hf(Y)Z_{p})+h[k]_{p, q}f(Y)Y^{k-1}Z_{p}Y\\
&=q^{k+1}Y^{k+1}X+q^{k}hY^kf(Y)Z_{p}+h[k]_{p, q}f(Y)Y^{k-1}Z_{p}Y\\
&=q^{k+1}Y^{k+1}X+q^{k}hf(Y)Y^kZ_{p}+ph[k]_{p, q}f(Y)Y^kZ_{p}\\
&=q^{k+1}Y^{k+1}X+h(q^k+p[k]_{p, q})f(Y)Y^{k}Z_{p}\\
&=q^{k+1}Y^{k+1}X+h[
k+1]_{p, q}f(Y)Y^kZ_p,
\end{align*}
where in the last equation we used that $[k+1]_{p, q}=q^k+p[k]_{p, q}$.
\end{proof}
Similarly, one can generalize the above lemma to a polynomial in $Y$ as follows.
\begin{proposition}\label{pqcrp}
Let $X, Y$ and $Z_p$ be variables satisfying \eqref{pqcr1} - \eqref{pqcr4}. For any polynomial $L\in\mathbb{C}[Y]$, we have the following identity,
\begin{equation}\label{polyrelation}
XL(Y)=L(qY)X+hf(Y)D_{p, q}L(Y)Z_p.
\end{equation}
\end{proposition}
\begin{proof}
Let $L(y)=\sum_{i=0}^{n}a_iy^i$ be a polynomial. By virtue of Lemma~\ref{lem1}, and using the linearity of the $(p, q)$-derivative operator $D_{p, q}$, one has
\begin{align*}
XL(Y)&=X\sum_{i=0}^{n}a_{i}Y^i\\
&=\sum_{i=0}^{n}a_i\big(q^iY^iX+h[i]_{p, q}f(Y)Y^{i-1}Z_p\big)\\
&=\sum_{i=0}^{n}a_i(qY)^iX+hf(Y)\sum_{i=0}^{n}a_i[i]_{p, q}Y^{i-1}Z_p\\
&=L(qY)X+hf(Y)D_{p, q}L(Y)Z_p,
\end{align*}
as requested.
\end{proof}
Now, letting $L\equiv f$ in Proposition~\ref{pqcrp}, we obtain
$$
Xf(Y)=f(qY)X+hf(Y)D_{p, q}f(Y)Z_p.
$$
Using this in \eqref{p2no} yields the normal ordered form of $X^2Y$,
\begin{equation}\label{p2no2}
X^2Y=q^2YX^2+h\big(qf(Y)+pf(qY)\big)Z_pX+h^2f(Y)D_{p, q}f(Y)Z_p^2.
\end{equation}

To describe the general case $X^mY$, let us first introduce some basic tools about partitions from \cite{TMMS2011} which we will use in the following. Recall that a {\em partition} $\lambda=(\lambda_1,\lambda_2,\ldots,\lambda_k) \equiv \lambda_1\lambda_2\ldots\lambda_k$ is a weakly decreasing sequence of positive integers. We denote the sum of the parts of $\lambda$ by $|\lambda|$, that is, $|\lambda|=\sum_{i=1}^k\lambda_i$. For a partition $\lambda$, the {\em Young diagram} $Y_\lambda$ of shape $\lambda$ is a left-justified diagram of $|\lambda|$ boxes, with $\lambda_i$ black boxes in the $i$th column. We denote the set of all Young diagrams that are contained in a $ \ell \times k$ box by ${\mathcal I}_{\ell,k}$. Define ${\mathcal I}_{m}=\bigcup_{k=0}^m {\mathcal I}_{m-k,k}$. Note that ${\mathcal I}_{m-k,k}$ contains $\binom{m}{k}$ elements, hence $|{\mathcal I}_m|=2^m$. In the following, we will denote the empty partition in ${\mathcal I}_{m-k,k}$ also by $\emptyset$ (instead of, e.g., $000 \in {\mathcal I}_{1,3}$ or $00 \in {\mathcal I}_{2,2}$, see  Figure~\ref{fig1}).

\begin{example} The set $\mathcal{I}_4$ containing 16 partitions is given in Figure~\ref{fig1}.
\begin{center}
\begin{figure}[htp]
\begin{pspicture}(14,2.2)
\psset{xunit=0.7,yunit=0.7}
\put(-.5,0){\psline(0,2.8)(1,2.8) \put(.2,2.15){$\mathcal{I}_{0,4}$}\put(0.8,1.9){$\lambda=0000$}}
\put(2,0){\put(0,.6){\psline(0,2)(.75,2)\psline(0,1.75)(.75,1.75)\psline(.25,1.75)(.25,2)\psline(.5,1.75)(.5,2)\psline(.75,1.75)(.75,2)
\psline(0,1.75)(0,2)}\put(0.8,2.15){$\mathcal{I}_{1,3}$}\put(.75,1.8){$\lambda=000$}}
\put(2,-.5){\put(0,.6){\psline(0,2)(.75,2)\psline(0,1.75)(.75,1.75)\psline(.25,1.75)(.25,2)\psline(.5,1.75)(.5,2)\psline(.75,1.75)(.75,2)
\psline(0,1.75)(0,2)\psline[fillstyle=solid,fillcolor=black](0,2)(.25,2)(.25,1.75)(0,1.75)(0,2)}\put(.75,1.8){$\lambda=100$}}
\put(2,-1){\put(0,.6){\psline(0,2)(.75,2)\psline(0,1.75)(.75,1.75)\psline(.25,1.75)(.25,2)\psline(.5,1.75)(.5,2)\psline(.75,1.75)(.75,2)
\psline(0,1.75)(0,2)\psline[fillstyle=solid,fillcolor=black](0,2)(.5,2)(.5,1.75)(0,1.75)(0,2)}\put(.75,1.8){$\lambda=110$}}
\put(2,-1.5){\put(0,.6){\psline(0,2)(.75,2)\psline(0,1.75)(.75,1.75)\psline(.25,1.75)(.25,2)\psline(.5,1.75)(.5,2)\psline(.75,1.75)(.75,2)
\psline(0,1.75)(0,2)\psline[fillstyle=solid,fillcolor=black](0,2)(.75,2)(.75,1.75)(0,1.75)(0,2)}\put(.75,1.8){$\lambda=111$}}
\put(4.5,0){\put(0,.6){\psline(0,2)(.5,2)(.5,1.5)(0,1.5)(0,2)\psline(0,1.75)(.5,1.75)\psline(.25,1.5)(.25,2)}
\put(1,2.15){$\mathcal{I}_{2,2}$}\put(.5,1.7){$\lambda=00$}}
\put(4.5,-.75){\put(0,.6){\psline(0,2)(.5,2)(.5,1.5)(0,1.5)(0,2)\psline(0,1.75)(.5,1.75)\psline(.25,1.5)(.25,2)
\psline[fillstyle=solid,fillcolor=black](0,2)(.25,2)(.25,1.75)(0,1.75)(0,2)}\put(.5,1.7){$\lambda=10$}}
\put(4.5,-1.5){\put(0,.6){\psline(0,2)(.5,2)(.5,1.5)(0,1.5)(0,2)\psline(0,1.75)(.5,1.75)\psline(.25,1.5)(.25,2)
\psline[fillstyle=solid,fillcolor=black](0,2)(.5,2)(.5,1.75)(0,1.75)(0,2)}\put(.5,1.7){$\lambda=11$}}
\put(6.25,0){\put(0,.6){\psline(0,2)(.5,2)(.5,1.5)(0,1.5)(0,2)\psline(0,1.75)(.5,1.75)\psline(.25,1.5)(.25,2)
\psline[fillstyle=solid,fillcolor=black](0,2)(.25,2)(.25,1.5)(0,1.5)(0,2)}\put(.5,1.7){$\lambda=20$}}
\put(6.25,-.75){\put(0,.6){\psline(0,2)(.5,2)(.5,1.5)(0,1.5)(0,2)\psline(0,1.75)(.5,1.75)\psline(.25,1.5)(.25,2)
\psline[fillstyle=solid,fillcolor=black](0,2)(.5,2)(.5,1.75)(.25,1.75)(.25,1.5)(0,1.5)(0,2)}\put(.5,1.7){$\lambda=21$}}
\put(6.25,-1.5){\put(0,.6){\psline(0,2)(.5,2)(.5,1.5)(0,1.5)(0,2)\psline(0,1.75)(.5,1.75)\psline(.25,1.5)(.25,2)
\psline[fillstyle=solid,fillcolor=black](0,2)(.5,2)(.5,1.5)(0,1.5)(0,2)}\put(.5,1.7){$\lambda=22$}}
\put(8.5,0){\put(0,.6){\psline(0,2)(.25,2)(.25,1.25)(0,1.25)(0,2)\psline(0,1.75)(.25,1.75)\psline(0,1.5)(.25,1.5)}
\put(1,2.15){$\mathcal{I}_{3,1}$}\put(.25,1.5){$\lambda=0$}}
\put(8.5,-1){\put(0,.6){\psline(0,2)(.25,2)(.25,1.25)(0,1.25)(0,2)\psline(0,1.75)(.25,1.75)\psline(0,1.5)(.25,1.5)
\psline[fillstyle=solid,fillcolor=black](0,2)(.25,2)(.25,1.75)(0,1.75)(0,2)}\put(.25,1.5){$\lambda=1$}}
\put(10.25,0){\put(0,.6){\psline(0,2)(.25,2)(.25,1.25)(0,1.25)(0,2)\psline(0,1.75)(.25,1.75)\psline(0,1.5)(.25,1.5)
\psline[fillstyle=solid,fillcolor=black](0,2)(.25,2)(.25,1.5)(0,1.5)(0,2)}\put(.25,1.5){$\lambda=2$}}
\put(10.25,-1){\put(0,.6){\psline(0,2)(.25,2)(.25,1.25)(0,1.25)(0,2)\psline(0,1.75)(.25,1.75)\psline(0,1.5)(.25,1.5)
\psline[fillstyle=solid,fillcolor=black](0,2)(.25,2)(.25,1.25)(0,1.25)(0,2)}\put(.25,1.5){$\lambda=3$}}
\put(12.7,0){\put(0,.6){\psline(0,2)(0,1)}\put(-.3,2.15){$\mathcal{I}_{4,0}$}\put(.1,1.25){$\lambda=0$}}
\end{pspicture}
\caption{The Young diagrams in $\mathcal{I}_4$.}\label{fig1}
\end{figure}
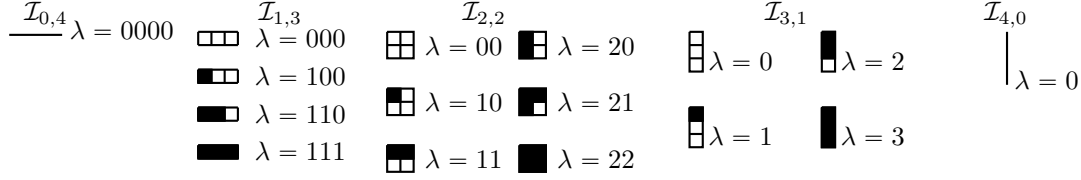
\end{center}
\end{example}

\begin{definition}\label{dgd}
Let $f(y)=\sum_{k=0}^s \alpha_k y^k$ be a polynomial. For a partition $\lambda$ with $d\geq 1$ parts, we define
\begin{align*}
g_{d}(\lambda; Y)&\equiv g_{d}(\lambda_1, \lambda_2, \ldots, \lambda_d; Y)\\
&:=\sum_{i_1, i_2, \ldots, i_d=0}^{s}\Big(q^{\sum_{j=1}^{d}i_j\lambda_j}\prod_{j=1}^{d}\alpha_{i_j}\prod_{j=1}^{d-1}[i_1+\cdots +i_j+1-j]_{p, q}\Big)Y^{i_1+i_2+\cdots +i_d+1-d}.
\end{align*}
Moreover, for $d=0$, we set
$$g_0(\emptyset; Y):=Y.$$
\end{definition}
Note that $g_{d}(\lambda; Y)$ depends on the polynomial $f$, but to simplify the notation, we do not indicate this dependency explicitly. It is clear that, for $d=1$, we obtain, for $k\geq 0$,
$$g_1(k; Y)=\sum_{i=0}^{s}q^{ki}\alpha_iY^i=f(q^kY).$$

\begin{example}\label{Explicitg} Let us give the explicit values of $g_{d}(\lambda; Y)$ for small $d$. 
\begin{itemize}
\item $g_{1}(\emptyset; Y)=f(Y)$,
\item $g_{1}(1; Y)=f(qY)$,
\item $g_{1}(2; Y)=f(q^2Y)$,
\item $g_{2}(\emptyset; Y)=f(Y)D_{p, q}f(Y)$,
\item $g_{2}(1; Y)=D_{p, q}f(qY)$,
\item $g_{2}(10; Y)=f(Y)D_{p, q}f(qY)$,
\item $g_{2}(11; Y)=f(qY)D_{p, q}f(qY)$,
\item $g_{3}(\emptyset; Y)=f(Y)\bigl(f(pY)D^2_{p, q}f(Y)+q^{-1}D_{p, q}f(qY)D_{p, q}f(Y)\bigr)$.
\end{itemize}
\end{example}

Following \cite{TMMS2011}, for any partition $\lambda =(\lambda_1,\lambda_2, \ldots, \lambda_d)$ we define $\lambda^{*}$ to be the partition
$$
\lambda^{*}=(\lambda_1+1,\lambda_2+1, \ldots, \lambda_d+1),
$$
and $\lambda_{*}$ to be the partition
$$
\lambda_{*}=(\lambda_1,\lambda_2, \ldots, \lambda_d,0).
$$
Thus, $Y_{\lambda^*}$ is the Young diagram that is obtained from $Y_{\lambda}$ by adding a filled row before the first row of $Y_{\lambda}$, and $Y_{\lambda_*}$ is the Young diagram that is obtained from $Y_{\lambda}$ by adding an empty column on the right side of $Y_{\lambda}$. Thus, if $Y_{\lambda} \in {\mathcal I}_{\ell,k}$, then $Y_{\lambda^*} \in {\mathcal I}_{\ell+1,k}$ and $Y_{\lambda_*} \in {\mathcal I}_{\ell,k+1}$.
\begin{lemma}\label{lxgd}
Let $\lambda$ be a partition with $d\geq 0$ parts. Let $X, Y$ and $Z_p$ be variables satisfying \eqref{pqcr1} - \eqref{pqcr4}, where $f(y)=\sum_{i=0}^{s}\alpha_i y^i$ is a polynomial. Then $g_{d}(\lambda; Y)$ satisfies the following recursive formula,
\begin{equation}\label{xgde}
Xg_{d}(\lambda; Y)=q^{1-d}g_{d}(\lambda^{*}; Y)X+hg_{d+1}(\lambda_{*}; Y)Z_p.
\end{equation}
\end{lemma}
\begin{proof}
The proof is the same as for $p=1$ given in \cite{TMMS2011}. By virtue of Definition~\ref{dgd}, we obtain
\begin{align*}
Xg_{d}(\lambda; Y)&=Xg_{d}(\lambda_1, \ldots, \lambda_d; Y)\\
&=\sum_{i_1, \ldots, i_d\geq 0}\Big(q^{\sum_{j=1}^{d}i_j\lambda_j}\prod_{j=1}^{d}\alpha_{i_j}\prod_{j=1}^{d-1}[i_1+\cdots +i_j+1-j]_{p, q}\Big) XY^{i_1+\cdots +i_d+1-d}\\
&=q^{1-d}\sum_{i_1, \ldots, i_d\geq 0}\Big(q^{\sum_{j=1}^{d}i_j\lambda^{*}_j}\prod_{j=1}^{d}\alpha_{i_j}\prod_{j=1}^{d-1}[i_1+\cdots +i_j+1-j]_{p, q}\Big)Y^{i_1+\cdots +i_d+1-d}X\\
&\qquad+ h\sum_{i_1, \ldots, i_{d+1}\geq 0}\Big(q^{\sum_{j=1}^{d}i_j\lambda_j}\prod_{j=1}^{d}\alpha_{i_j}\prod_{j=1}^{d}[i_1+\cdots +i_j+1-j]_{p, q}\Big)Y^{i_1+\cdots +i_d-d}Z_p\\
&=q^{1-d}g_{d}(\lambda^{*}; Y)X+hg_{d+1}(\lambda_{*}; Y)Z_p,
\end{align*}
where we used Lemma~\ref{lem1} in the second equality.
\end{proof}
Now, we are in the position to present the normal ordered form of $X^mY$.
\begin{theorem}
Let $X, Y$ and $Z_p$ be variables satisfying \eqref{pqcr1} - \eqref{pqcr4}, where $f(y)=\sum_{j=0}^{s}\alpha _{j}y^j$ is a polynomial. Then the normal ordered form of $X^mY$ is given by
\begin{equation}\label{noxmy}
X^mY=\sum_{j=0}^{m}h^j\Big(\sum_{\lambda\in \mathcal{I}_{m-j, j}}q^{m-j-\vert\lambda\vert}p^{\vert \lambda \vert}g_{j}(\lambda; Y)\Big)Z_{p}^jX^{m-j}.
\end{equation}
\end{theorem}
\begin{proof}
We use induction on $m$. Obviously, for $m=0$, the left-hand side of \eqref{noxmy} gives $Y$, and the right-hand side is given by
$$\sum_{\lambda\in \mathcal{I}_{0, 0}}q^{0-0-\vert\lambda\vert}p^{\vert\lambda\vert}g_{0}(\lambda; Y)=g_{0}(\emptyset; Y)=Y.$$
Assume that \eqref{noxmy} holds for $m \geq 0$. Then, by the induction hypothesis, we have
\begin{align*}
X^{m+1}Y&=\sum_{j=0}^{m}h^{j}\left(\sum_{\lambda\in \mathcal{I}_{m-j, j}}q^{m-j-\vert\lambda\vert}p^{\vert\lambda\vert}\big(Xg_{j}(\lambda; Y)\big)\right)Z_{p}^{j}X^{m-j}\\
&=\sum_{j=0}^{m}h^{j}\left(\sum_{\lambda\in \mathcal{I}_{m-j, j}}q^{m-j-\vert\lambda\vert}p^{\vert\lambda\vert}\Big(q^{1-j}g_{j}(\lambda^{*}; Y)X+hg_{j+1}(\lambda_{*}; Y)Z_p\Big)\right) Z_{p}^{j}X^{m-j}\\
&=\sum_{j=0}^{m}h^{j}\left(\sum_{\lambda\in \mathcal{I}_{m-j, j}}q^{m-j+1-\sum_{i=1}^{j}(\lambda_i+1)}p^{\vert\lambda\vert}g_{j}(\lambda^{*}; Y)XZ_{p}^{j}\right)X^{m-j}\\
&\qquad\qquad +\sum_{j=0}^{m}h^{j+1}\left(\sum_{\lambda\in \mathcal{I}_{m-j, j}}q^{m-j-\vert\lambda\vert}p^{\vert\lambda\vert}g_{j+1}(\lambda_{*}; Y)\right)Z_{p}^{j+1}X^{m-j}\\
&=\sum_{j=0}^{m}h^{j}\left(\sum_{\lambda\in \mathcal{I}_{m-j, j}}q^{m-j+1-\sum_{i=1}^{j}(\lambda_j+1)}p^{\vert\lambda\vert+j}g_{j}(\lambda^{*}; Y)\right)Z_{p}^{j}X^{m+1-j}\\
&\qquad\qquad +\sum_{j=0}^{m+1}h^{j}\left(\sum_{\lambda_{*}\in \mathcal{I}_{m+1-j, j}}q^{m+1-j-\vert\lambda_{*}\vert}p^{\vert\lambda_{*}\vert}g_{j}(\lambda_{*}; Y)\right)Z_{p}^{j}X^{m+1-j}\\
&=\sum_{j=0}^{m+1}h^{j}\left(\sum_{\lambda^{*}\in \mathcal{I}_{m+1-j, j}}q^{m+1-j-\vert\lambda^{*}\vert}p^{\vert\lambda^{*}\vert}g_{j}(\lambda^{*}; Y)\right)Z_{p}^{j}X^{m+1-j}\\
&\qquad\qquad +\sum_{j=0}^{m+1}h^{j}\left(\sum_{\lambda_{*}\in \mathcal{I}_{m+1-j, j}}q^{m+1-j-\vert\lambda_{*}\vert}p^{\vert\lambda_{*}\vert}g_{j}(\lambda_{*}; Y)\right)Z_{p}^{j}X^{m+1-j}\\
&=\sum_{j=0}^{m+1}h^{j}\left(\sum_{\lambda\in \mathcal{I}_{m+1-j, j}}q^{m+1-j-\vert\lambda\vert}p^{\vert\lambda\vert}g_{j}(\lambda; Y)\right)Z_{p}^{j}X^{m+1-j},
\end{align*}
where the second equality follows from Lemma~\ref{lxgd}. The sum over $\lambda_{*}$ is understood as the sum over all Young diagrams in $\mathcal{I}_{m+1-j, j}$, where the last column is empty (by convention, we define $\mathcal{I}_{k, -1}$ to be the empty set), while the sum over $\lambda^{*}$ is understood as the sum over all Young diagrams $\mathcal{I}_{m+1-j, j}$ such that the first row is filled.
\end{proof}
\begin{example}\label{Explicitno} Let us consider \eqref{noxmy} for small $m$, using the expressions given in Example~\ref{Explicitg}.
\begin{itemize}
\item For $m=1$, one obtains $\mathcal{I}_{1, 0}=\{\emptyset\}, \, \mathcal{I}_{0, 1}=\{\emptyset\}.$ Hence,
\begin{align*}
XY&=qg_{0}(\emptyset; Y)X+hg_{1}(\emptyset; Y)Z_p\\
&=qYX+hf(Y)Z_p,
\end{align*}
which agrees with the $(p, q)$-commutation relation \eqref{pqcr1}.
\item For $m=2$, we have $\mathcal{I}_{2, 0}=\{\emptyset\}, \, \mathcal{I}_{1, 1}=\{\emptyset, 1\}, \, \mathcal{I}_{0, 2}=\{\emptyset\}$. Hence,
\begin{align*}
X^2Y&=q^2g_{0}(\emptyset; Y)X^2+h\bigl(qg_{1}(\emptyset; Y)+pg_{1}(1; Y)\bigr)Z_pX+h^2g_{2}(\emptyset; Y)Z_p^2\\
&=q^2YX^2+h\big(qf(Y)+pf(qY)\big)Z_pX+h^2f(Y)D_{p, q}f(Y)Z_p^2,
\end{align*}
which agrees with \eqref{p2no2}.
\item For $m=3$, we have $\mathcal{I}_{3, 0}=\{\emptyset\}, \, \mathcal{I}_{2, 1}=\{\emptyset,1,2\}, \, \mathcal{I}_{1, 2}=\{\emptyset,10,11\}, \, \mathcal{I}_{0, 3}=\{\emptyset\}.$ Hence,
\begin{align*}
X^3Y=&q^3g_{0}(\emptyset; Y)X^3+h\left\{q^2g_{1}(\emptyset; Y)+qpg_{1}(1; Y)+p^2g_{1}(2; Y)\right\}Z_pX^2 \\ &+h^2\left\{qg_{2}(\emptyset; Y)+pg_{2}(10; Y)+q^{-1}p^2g_{2}(11; Y)\right\}Z_p^2X+h^3g_{3}(\emptyset; Y)Z_p^3\\ =&q^3 Y X^3+h\left\{q^2f(Y)+qpf(qY)+p^2f(q^2Y)\right\}Z_pX^2 \\ &+h^2\left\{q f(Y)D_{p,q}f(Y)+pf(Y)D_{p, q}f(qY)+q^{-1}p^2f(qY)D_{p, q}f(qY)\right\}Z_p^2X\\
&+h^3\left\{f(Y)\bigl(f(pY)D^2_{p, q}f(Y)+q^{-1}D_{p, q}f(qY)D_{p, q}f(Y)\bigr)\right\}Z_p^3.
\end{align*}
\end{itemize}
\end{example}
Now, let us specialize our considerations to the case where $f$ is a monomial.
\begin{corollary}
Let $f(y)=y^s$ be a monomial. Then $g_{d}(\lambda; Y)$ can be written as
\begin{equation}
g_{d}(\lambda; Y)=q^{s\vert\lambda\vert}\bigg(\prod_{j=1}^{d-1}[1+(s-1)j]_{p, q}\bigg)Y^{1+(s-1)d}.
\end{equation}
\end{corollary}

Hence, the normal ordered form of $X^mY$ given in \eqref{noxmy} can be simplified as follows,
\begin{align*}
X^mY&=\sum_{k=0}^{m}h^{k}\bigg(\sum_{\lambda\in \mathcal{I}_{m-k, k}}q^{m-k-\vert \lambda\vert}p^{\vert\lambda\vert}g_{k}(\lambda; Y)\bigg)Z_{p}^{k}X^{m-k}\\
&=\sum_{k=0}^{m}h^{k}\bigg(\sum_{\lambda\in \mathcal{I}_{m-k, k}}q^{m-k-\vert\lambda\vert}p^{\vert\lambda\vert}q^{s\vert\lambda\vert}\Big(\prod_{j=1}^{k-1}[1+(s-1)j]_{p, q}\Big)Y^{1+(s-1)k}\bigg)Z_{p}^{k}X^{m-k}\\
&=\sum_{k=0}^{m}h^{k}q^{m-k}\Big(\sum_{\lambda\in \mathcal{I}_{m-k, k}}q^{(s-1)\vert\lambda\vert}p^{\vert\lambda\vert}\Big)\Big(\prod_{j=1}^{k-1}[1+(s-1)j]_{p, q}\Big)Y^{1+(s-1)k}Z_{p}^{k}X^{m-k}.
\end{align*}

\begin{lemma}\label{lpq}
For $m, k\in\mathbb{N}$, such that $m\geq k$, we have
\begin{equation}
\sum_{\lambda\in \mathcal{I}_{m-k. k}}p^{\vert\lambda\vert}q^{(s-1)\vert\lambda\vert}={m\choose k}_{p, q^{s-1}}.
\end{equation}
\end{lemma}

Let us introduce the following definition in a similar way as in \cite[Definition 1.25]{TMMS2016}.
\begin{definition}\label{DefpqGenWeyl}
The {\em $(p, q)$-deformed generalized Weyl algebra $A_{s;h|p,q}$} is defined, for $s\in\mathbb{N}_{0}, h\in\mathbb{C}\setminus\{0\}$ and $p, q\in\mathbb{C}$, with $p^s=1$, as the complex unital algebra generated by variables $X, Y$ and $Z_p$ satisfying 
\begin{equation}\label{pqdefgenWeyl}
XY-qYX = h Y^sZ_p, \,\, Z_pY=pYZ_p, \,\, XZ_p=pZ_pX, \,\, Z_1 =I.
\end{equation}
\end{definition}
\begin{remark}
The condition $p^s=1$ is necessary to be consistent. Let us consider the word $XZ_pY$. We can write 
$$
XZ_pY=pZ_pXY=pZ_p(qYX+hY^sZ_p)=pqZ_pYX+phZ_pY^sZ_p=p^2qYZ_pX+p^{s+1}hY^sZ_p^2.
$$
On the other hand, we can also write
$$
XZ_pY=pXYZ_p=p(qYX+hY^sZ_p)Z_p=pqYXZ_p+phY^sZ_p^2=p^2qYZ_pX+phY^sZ_p^2.
$$
Thus, the two results coincide only if $p^s=1$.
\end{remark}
Note that, for $p=1$, the above definition reduces to the $q$-deformed generalized Weyl algebra $A_{s;h|1,q}=A_{s;h|q}$ \cite[Definition 1.25]{TMMS2016}. The algebra $A_{s;h|q}$ was introduced in \cite{TMMS2011}. The study of the case $p=q=h=1$, was initiated by Burde \cite{Burde2005} (for matrices $X$ and $Y$) and Varvak \cite{AV2005}, see \cite[Section 9.1]{TMMS2016} for more references. 

By virtue of Lemma~\ref{lpq}, we obtain the following result (note that we added in \eqref{xmyone} in the product the factor $[1]_{p, q}=1$ corresponding to $j=0$).

\begin{corollary}\label{xmycor} Let $X, Y$ and $Z_p$ be variables satisfying the commutation relations \eqref{pqdefgenWeyl} of the $(p, q)$-deformed generalized Weyl algebra $A_{s;h|p,q}$. Then the normal ordered form of $X^mY$, where $m \in \mathbb{N}$, is given by
\begin{equation}\label{xmyone}
X^mY=\sum_{k=0}^{m}\bigg(h^{k}q^{m-k}{m\choose k}_{p, q^{s-1}}\prod_{j=0}^{k-1}[1+(s-1)j]_{p, q}\bigg)Y^{1+(s-1)k}Z_{p}^{k}X^{m-k}.
\end{equation}
\end{corollary}

\subsection{The normal ordered form of $X^{m}Y^n$}
Formula \eqref{xmyone} can be extended to the case $X^mY^n$ with arbitrary $n \in \mathbb{N}$ as follows.

\begin{proposition}
Let $X, Y$ and $Z_p$ be variables satisfying the commutation relations \eqref{pqdefgenWeyl} of the $(p, q)$-deformed generalized Weyl algebra $A_{s;h|p,q}$. Then the normal ordered form of $X^mY^n$, where $m,n\in \mathbb{N}$, is given by
\begin{equation}\label{xmyn}
X^mY^n=\sum_{k=0}^{m} \bigg( h^kq^{n(m-k)} {m\choose k}_{p, q^{s-1}} \prod_{j=0}^{k-1}[n+(s-1)j]_{p, q}\bigg) Y^{n+(s-1)k}Z_{p}^{k}X^{m-k}.
\end{equation}
\end{proposition}
\begin{proof}
We use induction on $m\geq 0$. It is clear that \eqref{xmyn} holds true for $m=0$. Now, we assume that \eqref{xmyn} holds true for $m \geq 0$. Using $X^{m+1}Y^n=X(X^mY^n)$, this yields
$$
X^{m+1}Y^n=\sum_{k=0}^{m}h^kq^{n(m-k)}\sum_{\lambda\in \mathcal{I}_{m-k, k}}q^{(s-1)\vert\lambda\vert}p^{\vert\lambda\vert}\prod_{j=0}^{k-1}[n+(s-1)j]_{p, q}XY^{n+(s-1)k}Z_{p}^{k}X^{m-k},
$$
where we also used Lemma~\ref{lpq}. Applying \eqref{eq1}, we observe that
\begin{align*}
XY^{n+(s-1)k}Z_{p}^{k}X^{m-k}=&q^{n+(s-1)k}Y^{n+(s-1)k}XZ_{p}^{k}X^{m-k}\\
&+h[n+(s-1)k]_{p, q}Y^sY^{n+(s-1)k-1}Z_{p}^{k+1}X^{m-k}.
\end{align*}
It follows that
\begin{align*}
X^{m+1}Y^n=& \sum_{k=0}^{m}h^{k}q^{n(m-k)}\sum_{\lambda\in \mathcal{I}_{m-k, k}}q^{(s-1)\vert\lambda\vert}p^{\vert\lambda\vert}\prod_{j=0}^{k-1}[n+(s-1)j]_{p, q}q^{n+(s-1)k}Y^{n+(s-1)k}XZ_{p}^{k}X^{m-k}\\
& +\sum_{k=0}^{m}h^kq^{n(m-k)}\sum_{\lambda\in \mathcal{I}_{m-k, k}}q^{(s-1)\vert\lambda\vert}p^{\vert\lambda\vert}\prod_{j=0}^{k-1}[n+(s-1)j]_{p, q}h[n+(s-1)k]_{p, q}Y^sY^{n+(s-1)k-1}Z_{p}^{k+1}X^{m-k}\\
=&\sum_{k=0}^{m}h^kq^{n(m+1-k)}\sum_{\lambda\in \mathcal{I}_{m-k, k}}q^{(s-1)(\vert\lambda\vert+k)}p^{k+\vert\lambda\vert}\prod_{j=0}^{k-1}[n+(s-1)j]_{p, q}Y^{n+(s-1)k}Z_{p}^{k}X^{m+1-k}\\
& +\sum_{k=0}^{m}h^{k+1}q^{n(m-k)}\sum_{\lambda\in \mathcal{I}_{m-k, k}}q^{(s-1)\vert\lambda\vert}p^{\vert\lambda\vert}\prod_{j=0}^{k}[n+(s-1)j]_{p, q}Y^{n+(s-1)(k+1)}Z_{p}^{k+1}X^{m-k}\\
=&\sum_{k=0}^{m+1}h^kq^{n(m+1-k)}\left(\sum_{\lambda^{*}\in \mathcal{I}_{m+1-k, k}}q^{(s-1)\vert\lambda^{*}\vert}p^{\vert\lambda^{*}\vert}\right)\prod_{j=0}^{k-1}[n+(s-1)j]_{p, q}Y^{n+(s-1)k}Z_{p}^{k}X^{m+1-k}\\
&+\sum_{k=0}^{m+1}h^{k}q^{n(m+1-k)}\left(\sum_{\lambda_{*}\in \mathcal{I}_{m+1-k, k}}q^{(s-1)\vert\lambda_{*}\vert}p^{\vert\lambda_{*}\vert}\right)\prod_{j=0}^{k-1}[n+(s-1)j]_{p, q}Y^{n+(s-1)k}Z_{p}^{k}X^{m+1-k}\\
=&\sum_{k=0}^{m+1}h^{k}q^{n(m+1-k)}\left(\sum_{\lambda\in \mathcal{I}_{m+1-k, k}}q^{(s-1)\vert\lambda\vert}p^{\vert\lambda\vert}\right)\prod_{j=0}^{k-1}[n+(s-1)j]_{p, q}Y^{n+k(s-1)}Z_{p}^{k}X^{m+1-k},
\end{align*}
as was to be shown.
\end{proof}
It is interesting to consider the coefficients $k=0,1$ in \eqref{xmyn} explicitly, yielding an expansion
\begin{equation}\label{xmynexpand}
X^mY^n= q^{nm} Y^nX^m+ hq^{n(m-1)} [m]_{p, q^{s-1}}[n]_{p, q} Y^{n+(s-1)}Z_{p}X^{m-1}+{\mathcal O}(h^2),
\end{equation}
where we abbreviated by ${\mathcal O}(h^2)$ the sum in \eqref{xmyn} containing the summands for $k \geq 2$. Thus, in lowest order in $h$, the variables $X $ and $Y$ behave as if they were $q$-commuting, but there are corrections with terms containing $h$ in higher order.
\begin{remark}
For $h=0$, the variables $X$ and $Y$ $q$-commute, $XY=qYX$. In that case, \eqref{xmyn} reduces to $X^mY^n=q^{nm}Y^nX^m$, as it should, see also \eqref{xmynexpand}. On the other hand, specializing \eqref{xmyn} to $p=1$, we recover \cite[Corollary 4.11]{TMMS2011} (see also \cite[Theorem 1]{CCG2017}).
\end{remark}
Above, we considered in Corollary~\ref{xmycor} the case $X^mY$. After having derived a formula for the general case $X^mY^n$ in \eqref{xmyn}, we consider the case $XY^n$ in the following example.
\begin{example}
Let $m=1$ and $n \in \mathbb{N}$ arbitrary. From \eqref{xmyn}, we obtain
\begin{align*}
X Y^n=&\sum_{k=0}^{1} \bigg( h^k q^{n(1-k)} {1 \choose k}_{p, q^{s-1}} \prod_{j=0}^{k-1}[n+(s-1)j]_{p, q}\bigg) Y^{n+(s-1)k}Z_{p}^{k}X^{1-k}\\ =& q^{n} Y^{n}X + h [n]_{p, q}Y^{n+(s-1)}Z_{p},
\end{align*}
which is exactly the contents of Lemma~\ref{lem1} for $f(y)=y^s$ (see also \cite[Proposition 3.1]{LO2021}).
\end{example}

\subsection{The special cases $s=0,1,2, 3$.}
In the following sections, we consider \eqref{xmyn} for $s \in \{0,1,2,3 \}$, corresponding to $(p,q)$-deformations of well-known structures.

\subsubsection{The $(p,q)$-deformed Weyl algebra $(s=0)$}\label{SectpqWeyl}
Let $s=0$. Then \eqref{pqdefgenWeyl} reduces to
\begin{equation}\label{pqWeyl}
XY-qYX = h Z_p, \,\, Z_pY=pYZ_p, \,\, XZ_p=pZ_pX, \,\, Z_1 =I,
\end{equation}
which is a kind of $(p,q)$-deformation of the Weyl algebra. In fact, the relations \eqref{pqWeyl} (for $h=1$) are exactly the defining relations of the algebra $H_{p^{-1},q}$ considered by Gaddis \cite{JG2016} in an algebraic fashion. Specializing the parameter $p$ further, one obtains other algebras considered before. For example, letting $p=q$, we obtain (for $h=1$) the relations $XY-qYX = Z, \,ZY=qYZ, \,XZ=qZX$ discussed by Kirkman and Small \cite{EKLS1993}. Letting $p=q^2$, we obtain (for $h=1$) the relations
\begin{equation}\label{Blumen}
XY-qYX = Z, \,\, ZY=q^2YZ, \,\, XZ=q^2ZX
\end{equation}
considered by Blumen \cite{SB2006}. Choosing $p=1$ and $h=1$, one recovers from \eqref{pqWeyl} the commutation relation $XY-qYX = I$ of the $q$-deformed Weyl algebra (see, e.g., \cite{TMMS2016}).

For $s=0$, \eqref{xmyn} reduces to
\begin{equation}\label{xmynWeyl}
X^mY^n=\sum_{k=0}^{m} \bigg( h^kq^{n(m-k)} {m\choose k}_{p, q^{-1}} \frac{[n]_{p, q}!}{[n-k]_{p, q}!}\bigg) Y^{n-k}Z_{p}^{k}X^{m-k},
\end{equation}
where we used $\displaystyle\prod_{j=0}^{k-1}[n-j]_{p, q}=\frac{[n]_{p, q}!}{[n-k]_{p, q}!}$. One should regard \eqref{xmynWeyl} as the central commutation relation of the $(p,q)$-deformation of the Weyl algebra. Specializing further to $p=1$, we use that
$$
{m\choose k}_{1, q^{-1}}={m\choose k}_{q^{-1}}=q^{k(k-m)}{m\choose k}_{q}
$$
as well as $\frac{[n]_{1,q}!}{[n-k]_{1,q}!}=\frac{[n]_{q}!}{[n-k]_{q}!}=[k]_{q}!{n\choose k}_{q}$ and $Z_1=I$ to simplify \eqref{xmynWeyl} as follows,
$$
X^mY^n=\sum_{k=0}^{\min(m, n)} h^kq^{(n-k)(m-k)}[k]_{q}!{m\choose k}_{q}{n\choose k}_{q} Y^{n-k}X^{m-k}.
$$
This is a well-known normal ordering result in the $q$-deformed Weyl algebra (see \cite[Proposition 7.13]{TMMS2016}), as expected. Note that we can write \eqref{xmynWeyl} in a similar form as
\begin{equation}\label{xmynWeyl2}
X^mY^n=\sum_{k=0}^{m} h^kq^{n(m-k)} [k]_{p,q}!{m\choose k}_{p, q^{-1}} {n\choose k}_{p, q}  Y^{n-k}Z_{p}^{k}X^{m-k}.
\end{equation}
We can use this to derive the structure constants of the $(p,q)$-deformed Weyl algebra. If we write the basis elements in normal ordered form as $Y^rZ_p^sX^t$, where $r,s,t \in \mathbb{N}_0$, we have to consider the product of two elements $(Y^{r_2}Z_p^{s_2}X^{t_2})\cdot (Y^{r_1}Z_p^{s_1}X^{t_1})$. Thus, using \eqref{xmynWeyl2} for $X^{t_2}Y^{r_1}$, we obtain
$$
(Y^{r_2}Z_p^{s_2}X^{t_2})\cdot (Y^{r_1}Z_p^{s_1}X^{t_1})=\sum_{k=0}^{t_2} h^kq^{r_1(t_2-k)} [k]_{p,q}!{t_2\choose k}_{p, q^{-1}} {r_1\choose k}_{p, q}  Y^{r_2}Z_p^{s_2}Y^{r_1-k}Z_{p}^{k}X^{t_2-k}Z_p^{s_1}X^{t_1}.
$$
From \eqref{pqcr2}, we obtain $Z_p^{s_2}Y^{r_1-k}=p^{s_2(r_1-k)}Y^{r_1-k}Z_p^{s_2}$, and from \eqref{pqcr3} that $X^{t_2-k}Z_p^{s_1}=p^{s_1(t_2-k)}Z_p^{s_1}X^{t_2-k}$. This shows the following result.

\begin{proposition}\label{structurepqWeyl} Let $X,Y$ and $Z_p$ satisfy the relations \eqref{pqWeyl} of the $(p,q)$-deformed Weyl algebra. For $r_i,s_i,t_i \in \mathbb{N}_0$, $i=1,2$, one has the normal ordering result
$$
(Y^{r_2}Z_p^{s_2}X^{t_2})\cdot (Y^{r_1}Z_p^{s_1}X^{t_1})=\sum_{k=0}^{t_2} \mathscr{C}_{p,q,h}(r_1,r_2,s_1,s_2,t_1,t_2|k) \, Y^{r_2+r_1-k}Z_p^{s_2+s_1+k}X^{t_2+t_1-k},
$$
where the structure constants are given by
$$
\mathscr{C}_{p,q,h}(r_1,r_2,s_1,s_2,t_1,t_2|k)=h^k p^{s_2r_1+s_1t_2-k(s_2+s_1)}q^{r_1(t_2-k)} [k]_{p,q}!{t_2\choose k}_{p, q^{-1}} {r_1\choose k}_{p, q}.
$$
\end{proposition}

\subsubsection{The $(p,q)$-deformed shift algebra $(s=1)$}\label{Sectpqshift}
Let $s=1$. Then \eqref{pqdefgenWeyl} reduces to
\begin{equation}\label{pqshift}
XY-qYX = h Y Z_p, \,\, Z_pY=pYZ_p, \,\, XZ_p=pZ_pX, \,\, Z_1 =I,
\end{equation}
but where we also have the restriction $p=1$. Thus, we recover the commutation relation $XY-qYX = h Y$ of the $q$-deformed shift algebra (see, e.g., \cite[Section 8.4.3]{TMMS2016} and the references given therein), and it seems not possible introduce a $(p,q)$-deformed shift algebra in this way!

For $s=1$ and $p=1$, \eqref{xmyn} reduces to
\begin{equation}\label{xmynShiftp}
X^mY^n=\sum_{k=0}^{m} \bigg( h^kq^{n(m-k)} {m\choose k} ([n]_{q})^{k}\bigg) Y^{n}X^{m-k}.
\end{equation}
Specializing further by letting $q=1$, this reduces to
$$
X^mY^n=\sum_{k=0}^{m} h^k {m\choose k} n^{k}Y^{n}X^{m-k}=Y^n (X+hn)^m.
$$
The formula $X^mY^n=Y^n (X+hn)^m$ is well known for the shift algebra in the conventions used here (see \cite{RAS1958, RMW1967, WW1974}, or also \cite[Section 8.4.3]{TMMS2016} and the references given therein). It is also easy to see directly: The commutation relation \eqref{pqcr1} reduces for $p=q=1$ and $f(y)=y$ to $XY-YX=hY$, or, $XY=Y(X+h)$. By induction, one easily shows that $XY^n=Y^n(X+hn)$ as well as $X^mY=Y(X+h)^m$. Combined, one obtains $X^mY^n=Y^n (X+hn)^m$. 

Note that we can write \eqref{xmynShiftp} equivalently as
\begin{equation}\label{xmynShiftp2}
X^mY^n=Y^n(q^nX+h[n]_q)^m.
\end{equation}

Using \eqref{xmynShiftp}, we can show the following analog to Proposition~\ref{structurepqWeyl}.
\begin{proposition}\label{structurepqshift} Let $X$ and $Y$ satisfy the commutation relation $XY-qYX = h Y$ of the $q$-deformed shift algebra. For $r_i,s_i,t_i \in \mathbb{N}_0$, $i=1,2$, one has the normal ordering result
$$
(Y^{r_2}X^{t_2})\cdot (Y^{r_1}X^{t_1})=\sum_{k=0}^{t_2} \mathscr{D}_{q,h}(r_1,r_2,t_1,t_2|k) \, Y^{r_2+r_1}X^{t_2+t_1-k},
$$
where the structure constants are given by
$$
\mathscr{D}_{q,h}(r_1,r_2,t_1,t_2|k)= h^k  q^{r_1(t_2-k)} {t_2\choose k}([r_1]_{q})^{k}.
$$
\end{proposition}

\subsubsection{The $(p,q)$-deformed Jordan plane $(s=2)$}\label{SectpqJordan}
Let $s=2$. Then \eqref{pqdefgenWeyl} reduces to 
\begin{equation}\label{pqJordan}
XY-qYX = h Y^2 Z_p, \,\, Z_pY=pYZ_p, \,\, XZ_p=pZ_pX, \,\, Z_1 =I,
\end{equation}
which is a kind of $(p,q)$-deformation of the Jordan plane (with $p^2=1$). Choosing $p=1$, one recovers the commutation relation of the $q$-deformed Jordan plane (see \cite[Section 7.4]{TMMS2016} -- where it is also called {\em $q$-meromorphic Weyl algebra} -- and the references therein).

For $s=2$, \eqref{xmyn} reduces to
\begin{equation}\label{xmynJordan}
X^mY^n=\sum_{k=0}^{m} \bigg( h^kq^{n(m-k)} {m\choose k}_{p, q} \prod_{j=0}^{k-1}[n+j]_{p, q}\bigg) Y^{n+k}Z_{p}^{k}X^{m-k}.
\end{equation}
Using $\displaystyle\prod_{j=0}^{k-1}[n+j]_{p,q}=[k]_q! {n+k-1\choose k}_{p,q}$, we can write this equivalently as
\begin{equation}\label{xmynJordan2}
X^mY^n=\sum_{k=0}^{m} \bigg( h^kq^{n(m-k)} {n+k-1\choose k}_{p,q} \frac{[m]_{p,q}!}{[m-k]_{p,q}!}\bigg) Y^{n+k}Z_{p}^{k}X^{m-k}.
\end{equation}
Specializing further to $p=1$, this reduces to
$$
X^mY^n=\sum_{k=0}^{m} \bigg( h^kq^{n(m-k)} {n+k-1\choose k}_{q} \frac{[m]_q!}{[m-k]_q!}\bigg) Y^{n+k}X^{m-k},
$$
a well-known result for the $q$-deformed Jordan plane (see \cite[Proposition 7.61]{TMMS2016} and references therein), as expected. Using \eqref{xmynJordan2}, we can show the following analog to Proposition~\ref{structurepqWeyl} and Proposition~\ref{structurepqshift}.
\begin{proposition}\label{structurepqJordan} Let $X,Y$ and $Z_p$ satisfy the relations \eqref{pqJordan} of the $(p,q)$-deformed Jordan plane. For $r_i,s_i,t_i \in \mathbb{N}_0$, $i=1,2$, one has the normal ordering result
$$
(Y^{r_2}Z_p^{s_2}X^{t_2})\cdot (Y^{r_1}Z_p^{s_1}X^{t_1})=\sum_{k=0}^{t_2} \mathscr{E}_{p,q,h}(r_1,r_2,s_1,s_2,t_1,t_2|k) \, Y^{r_2+r_1+k}Z_p^{s_2+s_1+k}X^{t_2+t_1-k},
$$
where the structure constants are given by (recall that $p^2=1$)
$$
\mathscr{E}_{p,q,h}(r_1,r_2,s_1,s_2,t_1,t_2|k)= h^k  p^{s_2r_1+s_1t_2+k(s_2-s_1)}  q^{r_1(t_2-k)} {r_1+k-1\choose k}_{p,q} \frac{[t_2]_{p,q}!}{[t_2-k]_{p,q}!}.
$$
\end{proposition}

Before closing this section, let us mention that recently a different kind of deformation of the Jordan plane has been introduced in  \cite{ABGB2022} where the two generators $\mathcal{X}$ and $\mathcal{Z}$ satisfy $\mathcal{Z}\mathcal{X}-\mathcal{X}\mathcal{Z}=\mathcal{Z}^2+\Delta I$ with a parameter $\Delta \in \mathbb{R}$.

\subsubsection{The case $s=3$}\label{Sects3}
In this case, \eqref{pqdefgenWeyl} reduces to
\begin{equation}\label{pqdefgens3Weyl}
XY-qYX = h Y^3Z_p, \,\, Z_pY=pYZ_p, \,\, XZ_p=pZ_pX, \,\, Z_1 =I,
\end{equation}
with $p^3=1$. Hence, \eqref{pqdefgens3Weyl} can be seen as the defining relations of the $(p, q)$-deformed generalized Weyl algebra $A_{3;h|p,q}$. For $s=3$, \eqref{xmyn} reduces 
\begin{align}
X^mY^n&=\sum_{k=0}^{m} \bigg( h^kq^{n(m-k)} {m\choose k}_{p, q^{2}} \prod_{j=0}^{k-1}[n+2j]_{p, q}\bigg) Y^{n+2k}Z_{p}^{k}X^{m-k}\nonumber\\
&=\sum_{k=0}^{m} \bigg( h^kq^{n(m-k)} {m\choose k}_{p, q^{2}} [n|-2]_{k|p, q}\bigg) Y^{n+2k}Z_{p}^{k}X^{m-k}\label{xmyns3},
\end{align}
where, for $z,x \in\mathbb{C}$ and $n\in\mathbb{N}$, $[x|z]_{n|p, q}=\prod\limits_{k=0}^{n-1}[x-kz]_{p, q}$ is the $(p, q)$-deformed generalized factorial with $[x|z]_{0|p, q}=1$. Using \eqref{xmyns3}, we can show the following analog to Proposition~\ref{structurepqWeyl}, Proposition~\ref{structurepqshift} and Proposition~\ref{structurepqJordan}.
\begin{proposition}\label{structurepqdgWeyl}
Let $X,Y$ and $Z_p$ satisfy the relations \eqref{pqdefgens3Weyl} of the $(p,q)$-deformed generalized Weyl algebra $A_{3;h|p,q}$. For $r_i,s_i,t_i \in \mathbb{N}_0$, $i=1,2$, one has the normal ordering result
$$
(Y^{r_2}Z_p^{s_2}X^{t_2})\cdot (Y^{r_1}Z_p^{s_1}X^{t_1})=\sum_{k=0}^{t_2} \mathscr{E}_{p,q,h}(r_1,r_2,s_1,s_2,t_1,t_2|k) \, Y^{r_2+r_1+2k}Z_p^{s_2+s_1+k}X^{t_2+t_1-k},
$$
where the structure constants are given by (recall that $p^3=1$)
$$
\mathscr{E}_{p,q,h}(r_1,r_2,s_1,s_2,t_1,t_2|k)= h^k  p^{s_2r_1+s_1t_2+k(2s_2-s_1)}  q^{r_1(t_2-k)} {t_2\choose k}_{p,q^2} [r_1|-2]_{k|p;q}.
$$
\end{proposition}
\section{Normal ordering arbitrary words in the $(p,q)$-deformed generalized Weyl algebra}\label{gno}

In this section, we consider normal ordering arbitrary words $\omega$ in $X$ and $Y$ which satisfy the commutation relations \eqref{pqdefgenWeyl} of the $(p, q)$-deformed generalized Weyl algebra $A_{s;h|p,q}$. Such a word can be written in the form
\begin{equation}\label{wordxy}
\omega=X^{m_r}Y^{n_r}\cdots X^{m_2}Y^{n_2}X^{m_1}Y^{n_1}
\end{equation}
with suitable exponents $m_k,n_k \in \mathbb{N}_0=\mathbb{N}\cup\{0\}$. The resulting normal ordering coefficients can be interpreted as generalized Stirling numbers and have a connection to rook numbers.

\subsection{Normal ordering arbitrary words} To start, we consider normal ordering $X^{m_2}Y^{n_2}X^{m_1}Y^{n_1}$. This can be done in two steps: First commuting $X^{m_1}$ and $Y^{n_1}$ using \eqref{xmyn}, then commuting $X^{m_2}$ and $Y^{n_1+*}$ using again \eqref{xmyn}. Let us abbreviate for the following
$$
a_{s;p,q}(n,k):=\prod_{j=0}^{k-1}[n+(s-1)j]_{p, q}.
$$
The above-mentioned first step yields
$$
X^{m_2}Y^{n_2}X^{m_1}Y^{n_1}=\sum_{k_1=0}^{m_1} \bigg( h^{k_1}q^{n_1(m_1-k_1)} {m_1\choose k_1}_{p, q^{s-1}} a_{s;p,q}(n_1,k_1) \bigg) X^{m_2}Y^{n_2+n_1+(s-1)k_1}Z_{p}^{k_1}X^{m_1-k_1}.
$$
Now, as second step we have to commute $X^{m_2}$ and $Y^{n_2+n_1+(s-1)k_1}$ using \eqref{xmyn},
\begin{align*}
X^{m_2}Y^{n_2+n_1+(s-1)k_1}=&\sum_{k_2=0}^{m_2} \bigg( h^{k_2}q^{(n_2+n_1+(s-1)k_1)(m_2-k_2)} {m_2\choose k_2}_{p, q^{s-1}}a_{s;p,q}(n_2+n_1+(s-1)k_1,k_2)\bigg) \\ & Y^{n_2+n_1+(s-1)(k_1+k_2)}Z_{p}^{k_2}X^{m_2-k_2}.
\end{align*}
Inserting this, one obtains
\begin{align*}
X^{m_2}Y^{n_2}&X^{m_1}Y^{n_1}= \sum_{k_1=0}^{m_1} \sum_{k_2=0}^{m_2} \bigg\{ h^{k_1+k_2}q^{n_1(m_1-k_1)+(n_2+n_1+(s-1)k_1)(m_2-k_2)} {m_1\choose k_1}_{p, q^{s-1}}{m_2\choose k_2}_{p, q^{s-1}}  \\ &  a_{s;p,q}(n_1,k_1) a_{s;p,q}(n_2+n_1+(s-1)k_1,k_2)\bigg\} Y^{n_2+n_1+(s-1)(k_1+k_2)}Z_{p}^{k_2}X^{m_2-k_2} Z_{p}^{k_1}X^{m_1-k_1}.
\end{align*}
Using $X^{m_2-k_2} Z_{p}^{k_1}= p^{k_1(m_2-k_2)}Z_{p}^{k_1}X^{m_2-k_2}$, one finds
$$
Y^{n_2+n_1+(s-1)(k_1+k_2)}Z_{p}^{k_2}X^{m_2-k_2} Z_{p}^{k_1}X^{m_1-k_1}=p^{k_1(m_2-k_2)} Y^{n_2+n_1+(s-1)(k_1+k_2)}Z_{p}^{k_1+k_2}X^{m_1+m_2-(k_1+k_2)}.
$$
To introduce some further abbreviations, let us write ${\bf n}=(n_2,n_1)$ and denote $|{\bf n}|=n_1+n_2$, and, similarly, for ${\bf m}$ and ${\bf k}$. Denoting also
$$
{\bf m\choose \bf k}_{p, q^{s-1}}={m_2\choose k_2}_{p, q^{s-1}}{m_1\choose k_1}_{p, q^{s-1}},
$$
we can write the above result more briefly as
\begin{align*}
X^{m_2}Y^{n_2}X^{m_1}Y^{n_1}=& \sum_{k_1=0}^{m_1} \sum_{k_2=0}^{m_2} \bigg\{ h^{|{\bf k}|} p^{k_1(m_2-k_2)} q^{n_1(m_1-k_1)+(n_2+n_1+(s-1)k_1)(m_2-k_2)} {\bf m\choose \bf k}_{p, q^{s-1}}  \\ &  a_{s;p,q}(n_1,k_1) a_{s;p,q}(n_2+n_1+(s-1)k_1,k_2)\bigg\} Y^{|{\bf n}|+(s-1)|{\bf k}|}Z_{p}^{|{\bf k}|}X^{|{\bf m}|-|{\bf k}|} . 
\end{align*}
We can iterate this to see how a recursive structure appears in the coefficients,
\begin{align*}
X^{m_3}Y^{n_3}&X^{m_2}Y^{n_2}X^{m_1}Y^{n_1}= \sum_{k_1=0}^{m_1} \sum_{k_2=0}^{m_2} \bigg\{ h^{k_1+k_2}q^{n_1(m_1-k_1)+(n_2+n_1+(s-1)k_1)(m_2-k_2)} {m_1\choose k_1}_{p, q^{s-1}}{m_2\choose k_2}_{p, q^{s-1}}  \\ &  a_{s;p,q}(n_1,k_1) a_{s;p,q}(n_2+n_1+(s-1)k_1,k_2)\bigg\} X^{m_3}Y^{n_3+n_2+n_1+(s-1)(k_1+k_2)}Z_{p}^{k_1+k_2}X^{m_2+m_1-(k_1+k_2)}.
\end{align*}
Now, we have to commute $X^{m_3}$ and $Y^{n_3n_2+n_1+(s-1)(k_1+k_2)}$ using \eqref{xmyn},
\begin{align*}
X^{m_3}&Y^{n_3+n_2+n_1+(s-1)(k_1+k_2)}=\sum_{k_3=0}^{m_3} \bigg( h^{k_3}q^{(n_3+n_2+n_1+(s-1)(k_1+k_2))(m_3-k_3)} {m_3\choose k_3}_{p, q^{s-1}} \\ & a_{s;p,q}(n_3+n_2+n_1+(s-1)(k_1+k_2),k_3)\bigg)  Y^{n_3+n_2+n_1+(s-1)(k_1+k_2+k_3)}Z_{p}^{k_3}X^{m_3-k_3}.
\end{align*}
Inserting this, using $X^{m_3-k_3}Z_p^{k_1+k_2}=p^{(k_1+k_2)(m_3-k_3)}Z_p^{k_1+k_2}X^{m_3-k_3}$, and using the notation introduced above (extended from $r=2$ to arbitrary $r$), we obtain for the right-hand side:
\begin{align*}
 \sum_{k_1=0}^{m_1} \sum_{k_2=0}^{m_2}\sum_{k_3=0}^{m_3} \bigg\{ & h^{|{\bf k}|} q^{n_1(m_1-k_1)+(n_1+n_2+(s-1)k_1)(m_2-k_2)+(n_1+n_2+n_3+(s-1)(k_1+k_2))(m_3-k_3)}{\bf m\choose \bf k}_{p, q^{s-1}}  \\ & p^{k_1(m_2-k_2)+(k_1+k_2)(m_3-k_3)} a_{s;p,q}(n_1,k_1) a_{s;p,q}(n_1+n_2+(s-1)k_1,k_2)) \\ & a_{s;p,q}(n_1+n_2+n_3+(s-1)(k_1+k_2),k_3) \bigg\} Y^{|{\bf n}|+(s-1)|{\bf k}|}Z_{p}^{|{\bf k}|}X^{|{\bf m}|-|{\bf k}|} .
\end{align*}
Let us introduce some further notations to get less clumsy formulas. Let us first abbreviate $({\bf 0},{\bf m})_r$ to be the set of all ${\bf k}=(k_1,\ldots,k_r) \in \mathbb{N}_0^r$ such that $0 \leq k_j \leq m_j $ for $j=1,\ldots,r$. Furthermore, let us introduce
\begin{align*}
{\mathcal Q}_{r|s}({\bf m},{\bf n},{\bf k}):=&\sum_{j=1}^r (m_j-k_j)\left\{\sum_{l=1}^{j}n_l+(s-1)\sum_{l=1}^{j-1}k_l\right\},\\
{\mathcal P}_r({\bf m},{\bf k}):=& \sum_{j=1}^{r-1}(m_{j+1}-k_{j+1})\sum_{l=1}^{j}k_l, \\
{\mathcal A}_{r|s;p,q}({\bf n},{\bf k}):=& \prod_{j=1}^r a_{s;p,q}\left(\sum_{l=1}^{j}n_l+(s-1)\sum_{l=1}^{j-1}k_l,k_j\right)=\prod_{j=1}^r \prod_{\ell =0}^{k_j-1}\left[\sum_{l=1}^{j}n_l+(s-1)\sum_{l=1}^{j-1}k_l+(s-1) \ell \right]_{p, q}.
\end{align*}
Then we can write the above result compactly as
\begin{align*}
X^{m_3}Y^{n_3}X^{m_2}Y^{n_2} X^{m_1}Y^{n_1}=\sum_{{\bf k}  \in ({\bf 0},{\bf m})_3} \bigg\{ & h^{|{\bf k}|} p^{{\mathcal P}_3({\bf m},{\bf k})}  q^{{\mathcal Q}_{3|s}({\bf m},{\bf n},{\bf k})}  {\bf m\choose \bf k}_{p, q^{s-1}} \\ & {\mathcal A}_{3|s;p,q}({\bf n},{\bf k}) \bigg\} Y^{|{\bf n}|+(s-1)|{\bf k}|}Z_{p}^{|{\bf k}|}X^{|{\bf m}|-|{\bf k}|} .
\end{align*}
Thus, it is natural to conjecture that the above structure holds in general.
\begin{theorem}\label{thmword} Let $X, Y$ and $Z_p$ be variables satisfying the commutation relations \eqref{pqdefgenWeyl} of the $(p, q)$-deformed generalized Weyl algebra $A_{s;h|p,q}$ (with $p^s=1$). Then the normal ordered form of an arbitrary word $\omega=X^{m_r}Y^{n_r}\cdots X^{m_1}Y^{n_1}$, $m_k,n_k \in \mathbb{N}_0$,  is given by
\begin{equation}\label{Main}
\omega=\sum_{{\bf k}  \in ({\bf 0},{\bf m})_r} \bigg\{ h^{|{\bf k}|} p^{{\mathcal P}_r({\bf m},{\bf k})}  q^{{\mathcal Q}_{r|s}({\bf m},{\bf n},{\bf k})}  {\bf m\choose \bf k}_{p, q^{s-1}}  {\mathcal A}_{r|s;p,q}({\bf n},{\bf k}) \bigg\} Y^{|{\bf n}|+(s-1)|{\bf k}|}Z_{p}^{|{\bf k}|}X^{|{\bf m}|-|{\bf k}|}.
\end{equation}
\end{theorem}
\begin{proof} We give a proof by induction in $r\in \mathbb{N}$. Clearly, for $r=1,2,3$ the above formula holds true, as was discussed above. So it remains to show the formula for $r+1$, assuming it holds true for $r$. Thus, let us consider $
X^{m_{r+1}}Y^{n_{r+1}}X^{m_r}Y^{n_r}\cdots X^{m_1}Y^{n_1}$. Using the induction hypothesis, this equals
$$
\sum_{{\bf k}  \in ({\bf 0},{\bf m})_r} \bigg\{ h^{|{\bf k}|} p^{{\mathcal P}_r({\bf m},{\bf k})}  q^{{\mathcal Q}_{r|s}({\bf m},{\bf n},{\bf k})}  {\bf m\choose \bf k}_{p, q^{s-1}}  {\mathcal A}_{r|s;p,q}({\bf n},{\bf k}) \bigg\} X^{m_{r+1}}Y^{n_{r+1}+|{\bf n}|+(s-1)|{\bf k}|}Z_{p}^{|{\bf k}|}X^{|{\bf m}|-|{\bf k}|},
$$
where the bold vectors still have $r$ components. Now, one has to use \eqref{xmyn} to find
\begin{align*}
X^{m_{r+1}}&Y^{n_{r+1}+|{\bf n}|+(s-1)|{\bf k}|}=\sum_{k_{r+1}=0}^{m_{r+1}} \bigg( h^{k_{r+1}}q^{(n_{r+1}+|{\bf n}|+(s-1)|{\bf k}|)(m_{r+1}-k_{r+1})} {m_{r+1}\choose k_{r+1}}_{p, q^{s-1}} \\ & \prod_{j=0}^{k_{r+1}-1}[n_{r+1}+|{\bf n}|+(s-1)|{\bf k}|+(s-1)j]_{p, q}\bigg) Y^{n_{r+1}+|{\bf n}|+(s-1)|{\bf k}|+(s-1)k_{r+1}}Z_{p}^{k_{r+1}}X^{m_{r+1}-k_{r+1}}.
\end{align*}
Inserting this, using $X^{m_{r+1}-k_{r+1}}Z_p^{|{\bf k}|}=p^{|{\bf k}|(m_{r+1}-k_{r+1})}Z_p^{|{\bf k}|}X^{m_{r+1}-k_{r+1}}$, and collecting terms, one finds
\begin{align*}
\sum_{{\bf k}  \in ({\bf 0},{\bf m})_r}\sum_{k_{r+1}=0}^{m_{r+1}} & \bigg\{ h^{|{\bf k}|+k_{r+1}} p^{{\mathcal P}_r({\bf m},{\bf k})+|{\bf k}|(m_{r+1}-k_{r+1})} q^{{\mathcal Q}_{r|s}({\bf m},{\bf n},{\bf k})+(n_{r+1}+|{\bf n}|+(s-1)|{\bf k}|)(m_{r+1}-k_{r+1})}  {\bf m\choose \bf k}_{p, q^{s-1}} \\ & {m_{r+1}\choose k_{r+1}}_{p, q^{s-1}}     {\mathcal A}_{r|s;p,q}({\bf n},{\bf k}) a_{s;p,q}(n_{r+1}+|{\bf n}|+(s-1)|{\bf k}|,k_{r+1}) \bigg\} \\ & Y^{n_{r+1}+|{\bf n}|+(s-1)(|{\bf k}|+k_{r+1})}Z_{p}^{|{\bf k}|+k_{r+1}}X^{|{\bf m}|+m_{r+1}-(|{\bf k}|+k_{r+1})}
\end{align*}
Let us denote the vectors with $r+1$ components by ${\bf m}'=(m_{r+1},m_r,\ldots,m_1)=(m_{r+1},{\bf m})$, and, similarly, for ${\bf n}'$ and  ${\bf k}'$. The above definitions of the auxiliary functions show that
\begin{align*}
{\mathcal Q}_{r+1|s}({\bf m}',{\bf n}',{\bf k}')=&{\mathcal Q}_{r|s}({\bf m},{\bf n},{\bf k})+(m_{r+1}-k_{r+1})\left\{\sum_{l=1}^{r+1}n_l+(s-1)\sum_{l=1}^{r}k_l\right\}\\
=& {\mathcal Q}_{r|s}({\bf m},{\bf n},{\bf k})+(m_{r+1}-k_{r+1})(n_{r+1}+|{\bf n}|+(s-1)|{\bf k}|),\\
{\mathcal P}_{r+1}({\bf m}',{\bf k}')=&{\mathcal P}_{r}({\bf m},{\bf k}) +(m_{r+1}-k_{r+1})\sum_{l=1}^r k_l\\
=& {\mathcal P}_{r}({\bf m},{\bf k}) +(m_{r+1}-k_{r+1})|{\bf k}|, \\
{\mathcal A}_{r+1|s;p,q}({\bf n}',{\bf k}')=&{\mathcal A}_{r|s;p,q}({\bf n},{\bf k}) a_{s;p,q}\left(\sum_{l=1}^{r+1}n_l+(s-1)\sum_{l=1}^{r}k_l,k_{r+1}\right)\\
= &{\mathcal A}_{r|s;p,q}({\bf n},{\bf k}) a_{s;p,q}\left(n_{r+1}+|{\bf n}|+(s-1)|{\bf k}|,k_{r+1}\right).
\end{align*}
Observing $|{\bf n}'|=|{\bf n}|+n_{r+1}$ (and, similarly, for $|{\bf m}'|$ and  $|{\bf k}'|$), and using 
$${\bf m'\choose \bf k'}_{p, q^{s-1}}={m_{r+1}\choose k_{r+1}}_{p, q^{s-1}} {\bf m\choose \bf k}_{p, q^{s-1}}, 
$$ 
we can, therefore, write the preceding sum as
\begin{align*}
\sum_{{\bf k}'  \in ({\bf 0}',{\bf m}')_{r+1}} & \bigg\{ h^{|{\bf k}'|} p^{{\mathcal P}_{r+1}({\bf m}',{\bf k}')} q^{{\mathcal Q}_{r+1|s}({\bf m}',{\bf n}',{\bf k}')}  {\bf m'\choose \bf k'}_{p, q^{s-1}} {\mathcal A}_{r+1|s;p,q}({\bf n}',{\bf k}')\bigg\} Y^{|{\bf n}'|+(s-1)|{\bf k}'|}Z_{p}^{|{\bf k}'|}X^{|{\bf m}'|-|{\bf k}'|},
\end{align*}
as requested.
\end{proof}
By splitting the sum, we can write \eqref{Main} equivalently as
\begin{equation}\label{MainOther}
\omega=\sum_{k=0}^{|{\bf m}|} h^k \mathscr{J}_{p,q,s}({\bf m},{\bf n},k) Y^{|{\bf n}|+(s-1)k}Z_{p}^{k}X^{|{\bf m}|-k},
\end{equation}
where we introduced the coefficients
\begin{equation}\label{MainOtherCo}
\mathscr{J}_{p,q,s}({\bf m},{\bf n},k):=\sum_{{\bf k}  \in ({\bf 0},{\bf m})_r \atop |{\bf k}|=k} \bigg\{  p^{{\mathcal P}_r({\bf m},{\bf k})}  q^{{\mathcal Q}_{r|s}({\bf m},{\bf n},{\bf k})}  {\bf m\choose \bf k}_{p, q^{s-1}}  {\mathcal A}_{r|s;p,q}({\bf n},{\bf k}) \bigg\}.
\end{equation}
Thus, from \eqref{MainOther} we can infer that the variables in the normal ordered form appear in the combinations $Y^{|{\bf n}|}X^{|{\bf m}|}, Y^{|{\bf n}|+(s-1)}Z_{p}X^{|{\bf m}|-1},\ldots,Y^{|{\bf n}|+(s-1)|{\bf m}|}Z_{p}^{|{\bf m}|}$, with coefficients given by \eqref{MainOtherCo}. Note that the summand for $k=0$ is given by $\mathscr{J}_{p,q,s}({\bf m},{\bf n},0) Y^{|{\bf n}|}X^{|{\bf m}|}$ where 
$$
\mathscr{J}_{p,q,s}({\bf m},{\bf n},0) =p^{{\mathcal P}_r({\bf m},{\bf 0})}  q^{{\mathcal Q}_{r|s}({\bf m},{\bf n},{\bf 0})}   {\mathcal A}_{r|s;p,q}({\bf n},{\bf 0}).
$$
From the definitions, one finds 
$$
{\mathcal P}_r({\bf m},{\bf 0})=0, \,\,\, {\mathcal Q}_{r|s}({\bf m},{\bf n},{\bf 0})=\sum_{j=1}^r m_j \sum_{l=1}^j n_l, \,\,\, {\mathcal A}_{r|s;p,q}({\bf n},{\bf 0})=1.
$$
Let us define for any word $\omega=X^{m_r}Y^{n_r}\cdots X^{m_1}Y^{n_1}$ its {\em index} ${\mathcal I}(\omega)$ by 
$$
{\mathcal I}(\omega) :=\sum_{j=1}^r m_j \sum_{l=1}^j n_l,
$$
see \cite[Section 7.1.1]{TMMS2016}. If $X$ and $Y$ were $q$-commuting, it would count the number of times one has to use the commutation relation in order to bring $\omega$ into normal ordered form. Thus, ${\mathcal Q}_{r|s}({\bf m},{\bf n},{\bf 0})={\mathcal I}(\omega)$, and, consequently, $\mathscr{J}_{p,q,s}({\bf m},{\bf n},0) =q^{{\mathcal I}(\omega)}$. Therefore, we can write in analogy to \eqref{xmynexpand} the expansion
\begin{equation}\label{ExpGen}
\omega = q^{{\mathcal I}(\omega)} Y^{|{\bf n}|}X^{|{\bf m}|}+{\mathcal O}(h).
\end{equation}

\begin{corollary}
For $s=1$, one has the restriction $p=1$, see Section~\ref{Sectpqshift}. Thus, in the $q$-deformed shift algebra we have
$$
{\mathcal Q}_{r|1}({\bf m},{\bf n},{\bf k})=\sum_{j=1}^r (m_j-k_j)\sum_{l=1}^{j}n_l, \,\, {\mathcal A}_{r|1;1,q}({\bf n},{\bf k})= \prod_{j=1}^r a_{1;1,q}\left(\sum_{l=1}^{j}n_l\right)=\prod_{j=1}^r \prod_{\ell =0}^{k_j-1}\left[\sum_{l=1}^{j}n_l \right]_{q}.$$
Hence, \eqref{Main} reduces to
\begin{equation}
\omega=\sum_{{\bf k}  \in ({\bf 0},{\bf m})_r} \bigg\{ h^{|{\bf k}|} q^{\sum\limits_{j=1}^r (m_j-k_j)\sum\limits_{l=1}^{j}n_l}  {\bf m\choose \bf k}  \prod_{j=1}^r\left[\sum_{l=1}^{j}n_l \right]_{q}^{k_j} \bigg\} Y^{|{\bf n}|}X^{|{\bf m}|-|{\bf k}|}.
\end{equation}
In particular, we obtain the following extension of \eqref{xmynShiftp2},
$$\omega=Y^{|{\bf n}|}\prod_{j=1}^{r}\Biggl(q^{\sum\limits_{l=1}^{j}n_l}X+h\left[\sum_{l=1}^{j}n_l\right]_q\Biggr)^{m_j}.$$
\end{corollary}

\begin{corollary}
Recall that letting $h=0$ in \eqref{pqcr1} -- \eqref{pqcr4} means considering $q$-commuting variables $X$ and $Y$, i.e., $XY=qYX$. For $h=0$, \eqref{Main} reduces to $\omega=q^{{\mathcal I}(\omega)}Y^{|{\bf n}|}X^{|{\bf m}|}$, where we used the index of a word introduced above. This is a well-known result for $q$-commuting variables (see \cite[Section 7.1.1]{TMMS2016}). Comparing with \eqref{ExpGen}, we see that this is the lowest order term in an expansion in $h$.
\end{corollary}

\begin{corollary}\label{corpower} Let $X, Y$ and $Z_p$ be variables satisfying the commutation relations \eqref{pqdefgenWeyl} of the $(p, q)$-deformed generalized Weyl algebra $A_{s;h|p,q}$. Then the normal ordered form of the word $\omega=(XY)^r, r \in \mathbb{N}$, is given by
\begin{equation}\label{corpowereq}
(XY)^r=\sum_{k=0}^r h^{k}\bigg\{ \sum_{{\bf k}  \in \{0,1\}^r \atop |{\bf k}|=k}  p^{{\mathcal P}_r({\bf 1},{\bf k})}  q^{{\mathcal Q}_{r|s}({\bf 1},{\bf 1},{\bf k})}   {\mathcal A}_{r|s;p,q}({\bf 1},{\bf k}) \bigg\} Y^{r+(s-1)k}Z_{p}^{k}X^{r-k}.
\end{equation}
\end{corollary}
\begin{proof}
Using the fact that, for all ${\bf k}  \in \{0,1\}^r$, 
\begin{equation}\label{pqBinomRed}
 {\bf 1\choose \bf k}_{p, q^{s-1}} =1,
\end{equation}
this follows immediately from \eqref{Main} since in the case at hand ${\bf m}={\bf n}=(1,1,\ldots,1)$ of length $r$.
\end{proof}

In Theorem~\ref{thmword}, the normal ordered form of a word $\omega = w_{{\bf m},{\bf n}}=X^{m_r}Y^{n_r}\cdots X^{m_1}Y^{n_1}$, $m_k,n_k \in \mathbb{N}_0$, is given. What about considering the letter $Z_p$ on the left-hand side? Since $Z_p$ almost commutes with $X$ and $Y$ we can pull all letters $Z_p$ to the right (or, alternatively, to the left), yielding some powers of $p$ and something of the form $\omega^{\ast} Z_p^n$, where $\omega^{\ast}$ is a word in $X$ and $Y$ only. This can be normal ordered and then the factor $Z_p^n$ be brought to its correct position. Let us consider this in more detail and let 
\begin{equation}\label{PQWeylWord}
\mathrm{w}_{{\bf m},{\bf n},{\bf u}}=Z_p^{u_{2r}}Y^{n_r}Z_p^{u_{2r-1}}X^{m_r}\cdots Z_p^{u_4}Y^{n_2}Z_p^{u_3}X^{m_2}\cdot Z_p^{u_2}Y^{n_1}Z_p^{u_1}X^{m_1}\cdot Z_p^{u_0}
\end{equation}
be an arbitrary word in $X,Y$ and $Z_p$. Except for the right-most factor $Z_p^{u_0}$ we can group the factors as $r$ ``blocks'' of the form $Z_p^{u_{2k}}Y^{n_k}Z_p^{u_{2k-1}}X^{m_k}$, where $k=1,\ldots,r$. Commuting $Z_p^{u_{2k-1}}$ to the right yields a factor $p^{u_{2k-1}(\sum_{j=1}^{k-1}n_j-\sum_{j=1}^{k}m_j)}$, whereas commuting $Z_p^{u_{2k}}$ to the right yields a factor $p^{u_{2k}(\sum_{j=1}^{k}n_j-\sum_{j=1}^{k}m_j)}$. Together, we get as exponent 
$$
u_{2k-1}\left(\sum_{j=1}^{\lfloor \frac{2k-1}{2} \rfloor} n_j-\sum_{j=1}^{\lfloor \frac{2k}{2} \rfloor}m_j\right)+u_{2k}\left(\sum_{j=1}^{\lfloor \frac{2k}{2} \rfloor}n_j-\sum_{j=1}^{\lfloor \frac{2k+1}{2} \rfloor}m_j)\right).
$$
Thus, pulling all letters $Z_p$ of $\mathrm{w}_{{\bf m},{\bf n},{\bf u}}$ to the right, we get a factor of $p^{\mathscr{L}({\bf m},{\bf n},{\bf u})}$, where
\begin{equation}\label{PQExponent}
\mathscr{L}({\bf m},{\bf n},{\bf u}):=\sum_{\ell =1}^{2r} u_{\ell }\left(\sum_{j=1}^{\lfloor \frac{\ell }{2} \rfloor} n_j-\sum_{j=1}^{\lfloor \frac{\ell +1}{2} \rfloor}m_j\right).
\end{equation}
Denoting by $\mathrm{w}_{{\bf m},{\bf n},{\bf u}}^{\ast}$ the {\it reduced word of $\mathrm{w}_{{\bf m},{\bf n},{\bf u}}$} where all letters $Z_p$ are deleted from $\mathrm{w}_{{\bf m},{\bf n},{\bf u}}$, we can write (using  $|{\bf u}|=u_{2r}+u_{2r-1}+\cdots + u_1+u_0$)
\begin{equation}
\mathrm{w}_{{\bf m},{\bf n},{\bf u}}=p^{\mathscr{L}({\bf m},{\bf n},{\bf u})}\mathrm{w}_{{\bf m},{\bf n},{\bf u}}^{\ast} Z_p^{|{\bf u}|}.
\end{equation}
Now, we can use \eqref{Main} for $\mathrm{w}_{{\bf m},{\bf n},{\bf u}}^{\ast} $ to obtain the following result.
\begin{proposition}Let $X, Y$ and $Z_p$ be variables satisfying the commutation relations \eqref{pqdefgenWeyl} of the $(p, q)$-deformed generalized Weyl algebra $A_{s;h|p,q}$ (with $p^s=1$). Then the normal ordered form of an arbitrary word $\mathrm{w}_{{\bf m},{\bf n},{\bf u}}=Z_p^{u_{2r}}Y^{n_r}Z_p^{u_{2r-1}}X^{m_r}\cdots Z_p^{u_2}Y^{n_1}Z_p^{u_1}X^{m_1}Z_p^{u_0}$ is given by
\begin{equation}
\mathrm{w}_{{\bf m},{\bf n},{\bf u}}=\sum_{{\bf k}  \in ({\bf 0},{\bf m})_r} {\mathfrak N}_{p,q;h,s}({\bf m},{\bf n},{\bf u}| {\bf k}) \,  Y^{|{\bf n}|+(s-1)|{\bf k}|}Z_{p}^{|{\bf u}|+|{\bf k}|}X^{|{\bf m}|-|{\bf k}|},
\end{equation}
where we introduced the abbreviation 
$$
{\mathfrak N}_{p,q;h,s}({\bf m},{\bf n},{\bf u}|{\bf k}) :=\bigg\{ h^{|{\bf k}|} p^{{\mathcal P}_r^{\ast}({\bf m},{\bf k})+|{\bf u}|(|{\bf m}|-k)+\mathscr{L}({\bf m},{\bf n},{\bf u})}  q^{{\mathcal Q}_{r|s}^{\ast}({\bf m},{\bf n},{\bf k})}  {\bf m\choose \bf k}_{p, q^{s-1}}  {\mathcal A}_{r|s;p,q}^{\ast}({\bf n},{\bf k}) \bigg\}
$$
with $\mathscr{L}({\bf m},{\bf n},{\bf u})$ defined in \eqref{PQExponent} and the functions ${\mathcal P}_r^{\ast}, {\mathcal Q}_{r|s}^{\ast}$ and ${\mathcal A}_{r|s;p,q}^{\ast}$ have been labeled with the asterisk ``$\ast$'' to indicate that they are evaluated for the reduced word $\mathrm{w}_{{\bf m},{\bf n},{\bf u}}^{\ast}$.  
\end{proposition}

\subsection{The connection to generalized Stirling numbers}\label{SectGenStir}
Let us recall that in this context {\it generalized Stirling numbers} $\mathfrak{S}_{s;h}(n,k|p,q)$ were introduced as normal ordering coefficients of $(YX)^n$ in $A_{s;h|p,q}$ as follows (see \cite[Definition 3.3]{LO2021}),
\begin{equation}\label{StirGen}
(YX)^n=\sum_{k=0}^n \mathfrak{S}_{s;h}(n,k|p,q) \, Y^{s(n-k)+k}Z_p^{n-k} X^{k}.
\end{equation}
They satisfy the recurrence relation \cite[Theorem 3.5]{LO2021}
\begin{equation}\label{Stirlingrecu}
\mathfrak{S}_{s;h}(n+1,k|p,q)=p^{n-k+1}q^{s(n-k+1)+k-1}\mathfrak{S}_{s;h}(n,k-1|p,q)+h[s(n-k)+k]_{p, q}\mathfrak{S}_{s;h}(n,k|p,q).
\end{equation}
\begin{example}\label{pqStirling} Let $s=0$. Then \eqref{StirGen} reduces to 
\begin{equation}\label{StirGenS0}
(YX)^n=\sum_{k=0}^n \mathfrak{S}_{0;h}(n,k|p,q) \, Y^{k}Z_p^{n-k} X^{k}.
\end{equation}
Comparing this with \cite[Proposition 3]{LO2020}, we can identify $\mathfrak{S}_{0;1}(n,k|p,q)=S_{p,q}(n,k)$ with the $(p,q)$-Stirling numbers of the second kind discussed in \cite{LO2020}. One has $S_{p,q}(0,0)=1$ and 
\begin{equation}\label{pqStirlingRecu}
S_{p,q}(n,k)=p^{n-k}q^{k-1}S_{p,q}(n-1,k-1)+[k]_{p,q}S_{p,q}(n-1,k),
\end{equation}
following also from \eqref{Stirlingrecu} for $s=0$ and $h=1$. 
\end{example}
Since $(YX)^n=Y(XY)^{n-1}X$, we can use Corollary~\ref{corpower} for $r=n-1$ to obtain
$$
(YX)^n=\sum_{k=0}^{n-1} h^{k}\bigg\{ \sum_{{\bf k}  \in \{0,1\}^{n-1}\atop |{\bf k}|=k}  p^{{\mathcal P}_{n-1}({\bf 1},{\bf k})}  q^{{\mathcal Q}_{n-1|s}({\bf 1},{\bf 1},{\bf k})}   {\mathcal A}_{n-1|s;p,q}({\bf 1},{\bf k}) \bigg\} Y^{n+(s-1)k}Z_{p}^{k}X^{n-k}.
$$
Letting $\kappa = n-k$, this gives
$$
(YX)^n=\sum_{\kappa=1}^{n} h^{n-\kappa}\bigg\{ \sum_{{\bf k}  \in \{0,1\}^{n-1}\atop |{\bf k}|=n-\kappa}  p^{{\mathcal P}_{n-1}({\bf 1},{\bf k})}  q^{{\mathcal Q}_{n-1|s}({\bf 1},{\bf 1},{\bf k})}   {\mathcal A}_{n-1|s;p,q}({\bf 1},{\bf k}) \bigg\} Y^{s(n-\kappa)+\kappa}Z_{p}^{n-\kappa}X^{\kappa}.
$$
Comparing this with \eqref{StirGen}, one obtains the following result.
\begin{proposition} The generalized Stirling numbers $\mathfrak{S}_{s;h}(n,k|p,q)$ can be expressed as follows,
\begin{equation}\label{StirExpr}
\mathfrak{S}_{s;h}(n,k|p,q) =h^{n-k}\sum_{{\bf k}  \in\{0,1\}^{n-1}\atop |{\bf k}|=n-k}  p^{{\mathcal P}_{n-1}({\bf 1},{\bf k})}  q^{{\mathcal Q}_{n-1|s}({\bf 1},{\bf 1},{\bf k})}   {\mathcal A}_{n-1|s;p,q}({\bf 1},{\bf k}).
\end{equation}
\end{proposition}
Note that from \eqref{StirExpr} it is obvious that $\mathfrak{S}_{s;h}(n,k|p,q)=h^{n-k}\mathfrak{S}_{s;1}(n,k|p,q)$ (\cite[Corollary 3.6]{LO2021}). If we define the corresponding $(p,q)$-generalized Bell polynomials by 
\begin{equation}\label{PQBellpol}
\mathfrak{B}^{(s;h)}_{n|p,q}(x):=\sum_{k=0}^n \mathfrak{S}_{s;h}(n,k|p,q)x^k,
\end{equation}
this translates into
\begin{equation}\label{Bellscale}
\mathfrak{B}^{(s;h)}_{n|p,q}(x)=h^n \, \mathfrak{B}^{(s;1)}_{n|p,q}(x/h).
\end{equation}
The generalized Stirling numbers $\mathfrak{S}_{s;h}(n,k|p,q)$ are defined as normal ordering coefficients of the particular word $(YX)^n$, see \eqref{StirGen}. Following \cite{CCG2017}, we consider the normal ordering coefficients of arbitrary words. 
\begin{definition}\label{rooknum}
Let $X, Y$ and $Z_p$ be variables satisfying the commutation relations \eqref{pqdefgenWeyl} of the $(p, q)$-deformed generalized Weyl algebra $A_{s;h|p,q}$ (with $p^s=1$). The $(p, q)$-deformed generalized Stirling numbers $S_{s, h; p, q}^{{\bf m}, {\bf n}}[k]$ are defined as normal ordering coefficients of the string $\mathrm{w}_{{\bf m}, {\bf n}}=Y^{n_r}X^{m_r}\cdots Y^{n_1}X^{m_1}$, i.e., by
\begin{equation}\label{pqrook}
\mathrm{w}_{{\bf m}, {\bf n}}=\sum_{k=m_1}^{|{\bf m}|}S_{s, h; p, q}^{{\bf m}, {\bf n}}[k]Y^{|{\bf n}|-(|{\bf m}|-k)(1-s)}Z_{p}^{|{\bf m}|-k}X^{k}.
\end{equation}
\end{definition}

For the special case ${\bf m}={\bf 1}_r=(
\raisebox{0pt}[\height][0pt]{$\underbrace{1, 1\ldots1}_{r\ \text{times}})$}$ and ${\bf n}={\bf 1}_r=(
\raisebox{0pt}[\height][0pt]{$\underbrace{1, 1\ldots1}_{r\ \text{times}})$}$ we have
\vspace{.5cm}
\begin{equation}
(YX)^r=\sum_{k=1}^{r}S_{s, h; p, q}^{{\bf 1}_r,{\bf 1}_r}[k] Y^{s(r-k)+k}Z_{p}^{r-k}X^{k},
\end{equation}
which by comparison with \eqref{StirGen} shows that $S_{s, h; p, q}^{{\bf 1}_r,{\bf 1}_r}[k]=\mathfrak{S}_{s;h}(r,k|p,q)$, for $r \geq 1$. The expression of the following proposition for the numbers $S_{s, h; p, q}^{{\bf m}, {\bf n}}[k]$ provides an extension of \cite[Equation (7)]{CCG2017} (where $p=1$) to arbitrary $p$.
\begin{proposition}\label{pqrookfor}
The numbers $S_{s, h; p, q}^{{\bf m}, {\bf n}}[k]$ are given by
\begin{equation}\label{rookform2}
S_{s, h; p, q}^{{\bf m}, {\bf n}}[k]=h^{|{\bf m}|-k}\sum_{k_1+k_2+\cdots +k_{r-1}=|{\bf m}|-k}\prod_{i=1}^{r-1}\Psi_{s; p, q}\left[k_i, \sum_{j=1}^{i}n_j+(s-1)\sum_{j=1}^{i-1}k_j, \sum_{j=1}^{i-1}k_j, m_{i+1}\right],
\end{equation}
where
$$\Psi_{s; p, q}[j, k, l, m]=p^{(l-j)m}q^{k(l-j)} {l\choose j}_{p, q^{s-1}}\prod_{\ell=0}^{j-1}[k+\ell(s-1)]_{p, q}.$$
\end{proposition}
\begin{proof}
By virtue of Theorem~\ref{thmword}, the string $\mathrm{w}_{{\bf m}, {\bf n}}$ can be written as
\begin{align}\label{rookform1}
\mathrm{w}_{{\bf m}, {\bf n}}=\sum_{k_1=0}^{m_2}&\cdots \sum_{k_{r-1}}^{m_r}h^{k_1+\cdots +k_{r-1}}\prod_{i=1}^{r-1}\Psi_{s; p, q}\left[k_i, \sum_{j=1}^{i}n_j+(s-1)\sum_{j=1}^{i-1}k_j, \sum_{j=1}^{i-1}k_j, m_{i+1}\right]\\
&Y^{|{\bf n}|-(s-1)(k_1+\cdots+k_{r-1})}Z_{p}^{k_1+\cdots+k_{r-1}}X^{|{\bf m}|-(k_1+\cdots+k_{r-1})}.\nonumber
\end{align}
By comparing the coefficient of $X^k$ in \eqref{pqrook} and \eqref{rookform1}, we obtain \eqref{rookform2}.
\end{proof}
\begin{corollary}
For the special case ${\bf m}={\bf 1}_r=(
\raisebox{0pt}[\height][0pt]{$\underbrace{1, 1\ldots1}_{r\ \text{times}})$}$ and ${\bf n}={\bf 1}_r=(
\raisebox{0pt}[\height][0pt]{$\underbrace{1, 1\ldots1}_{r\ \text{times}})$}$ we have $S_{s, h; p, q}^{{\bf 1}_r,{\bf 1}_r}[k]=\mathfrak{S}_{s;h}(r,k|p,q)$, hence
$$
\mathfrak{S}_{s;h}(r,k|p,q)=h^{r-k}\sum_{k_1+\cdots +k_{r-1}=r-k}\quad\prod_{i=1}^{r-1}\Psi_{s; p, q}[k_i, i+(s-1)\sum_{j=1}^{i-1}k_j,\sum_{j=1}^{i-1}k_j,1].
$$
\end{corollary}
Generalizing \eqref{PQBellpol}, we can define the $(p, q)$-generalized Bell polynomials $B_{s, h; p, q}^{{\bf m}, {\bf n}}(x)$ and $(p, q)$-generalized Bell numbers $B_{s, h; p, q}^{{\bf m}, {\bf n}}$ associated to $S_{s, h; p, q}^{{\bf m}, {\bf n}}[k]$ in a straightforward fashion,
$$
B_{s, h; p, q}^{{\bf m}, {\bf n}}(x):=\sum_{k=m_1}^{|{\bf m}|}S_{s, h;p, q}^{{\bf m}, {\bf n}}[k]x^k, \,\, \,\, B_{s, h; p, q}^{{\bf m}, {\bf n}}:=B_{s, h; p, q}^{{\bf m}, {\bf n}}(1)=\sum_{k=m_1}^{|{\bf m}|}S_{s, h;p, q}^{{\bf m}, {\bf n}}[k].
$$

\section{Conclusion}\label{concl}
In this paper, we considered variables $X, Y$ and $Z_p$ satisfying the $(p, q)$-commutation relation $XY-qYX=hf(Y)Z_p$ with $XZ_p=pZ_pX$ and $Z_pY=pZ_pY$ for a polynomial $f$. Due to the variable $Z_p$ there are three variables which do not commute. However, $Z_p$ ``almost'' commutes with $X$ and $Y$, so the situation is ``close'' to the one where only $X$ and $Y$ exist and satisfy $XY-qYX=hf(Y)$. For most of the results we considered the special case where $f$ is a monomial, $f(Y)=Y^s$ with $s\in \mathbb{N}_0$. In this case, the algebraic structure generated by $X, Y$ and $Z_p$ satisfying the above commutation relations was called $(p,q)$-deformed generalized Weyl algebra (with $p^s=1$). We presented the normal ordering of arbitrary words in the letters $X,Y,Z_p$ and related aspects in the $(p,q)$-deformed generalized Weyl algebra. In particular, the relation to $(p,q)$-deformed generalized Stirling numbers was discussed in detail.

In a sequel \cite{TLM2}, we introduce $(p,q)$-rook numbers and give an interpretation of the normal ordering coefficients of words in the $(p,q)$-deformed generalized Weyl algebra in terms of these $(p,q)$-rook numbers. In further work \cite{TLM3}, the binomial formula for $(X+Y)^n$ is studied in the $(p,q)$-deformed generalized Weyl algebra.
\subsection*{Author Contributions}
The authors contributed equally to the preparation of this article. All authors read and approved the final manuscript.
\subsection*{Funding}
The authors declare that no funds, grants, or other support were received during the preparation of this manuscript.
\subsection*{Data Availability}
No datasets were generated or analysed during the current study.
\section*{Declarations}

\subsection*{Conflict of Interest}
The authors declare no conflict of interest.


\begin{thebibliography}{99}

\bibitem{AB2010} 
A.N.F. Aleixo and A.B. Balantekin, {\em The ladder operator normal ordering problem for quantum confined systems and the generalization of the Stirling and Bell numbers}, J. Phys. A, Math. Theor. {\bf 43} (2010), Article 045302.

\bibitem{AB2015} 
A.N.F. Aleixo and A.B. Balantekin, {\em Normal ordering for nonlinear deformed ladder operators and the f-generalization of the Stirling and Bell numbers}, J. Math. Phys. {\bf 56} (2015), Article 122108.

\bibitem{ABGB2022}
A. Beaudoin, G. Bergeron, A. Brillant, J Gaboriaud, L. Vinet and A. Zhedanov, \emph{Orthogonal polynomials and the deformed Jordan plane}, J. Math. Anal. Appl. {\bf 507} (2022), Article 125717.

\bibitem{BLO2013} 
G. Benkart, S.A. Lopes and M. Ondrus, {\em A parametric family of subalgebras of the Weyl algebra. II: Irreducible modules}, Recent developments in algebraic and combinatorial aspects of representation theory. Contemp. Math. 602, Amer. Math. Soc., Providence, RI (2013), 73--98.

\bibitem{BLO2015} 
G. Benkart, S.A. Lopes and M. Ondrus, {\em A parametric family of subalgebras of the Weyl algebra. I: Structure and automorphisms}, Trans. Am. Math. Soc. {\bf 367} (2015), 1993--2021.

\bibitem{BLO2015a} 
G. Benkart, S.A. Lopes and M. Ondrus, {\em Derivations of a parametric family of subalgebras of the Weyl algebra}, J. Algebra {\bf 424} (2015), 46--97.

\bibitem{BF2011} 
P. Blasiak and P. Flajolet, {\em Combinatorial models of creation-annihilation}, S{\'e}m. Lothar. Combin. {\bf 65} (2011), Article B65c.

\bibitem{SB2006}
S.C. Blumen, \emph{Two generalisations of the binomial theorem}, Austral. Math. Soc. Gaz. \textbf{33} (2006), 39--43.

\bibitem{BK1994}
I.M. Burban and A.U. Klimyk, \emph{$(P,Q)$-differentiation, $(P,Q)$-integration, and $(P,Q)$-hypergeometric functions related to quantum groups}, Integral Transforms Spec. Funct. {\bf 2} (1994), 15--36. 

\bibitem{Burde2005}
D. Burde, \emph{On the matrix equation $XA-AX = X^p$}, Linear Algebra Appl. {\bf 404} (2005), 147--165. 

\bibitem{CCG2017}
R.O. Celeste, R.B. Corcino and K.J.M. Gonzales, {\em Two approaches to normal order coefficients}, J. Integer Seq. {\bf 20} (2017), Art. 17.3.5.

\bibitem{JG2016}
J. Gaddis, \emph{Two-parameter analogs of the Heisenberg enveloping algebra},  Commun. Algebra \textbf{44} (2016), 4637--4653.

\bibitem{HOT2026}
A. Habib, A. Ouahhabi and E.H. Tahri, \emph{Operator-valued generalized Stirling numbers of the first kind from generalized boson algebra}, Phys. Lett. A {\bf 592} (2026), Article 131948.

\bibitem{Jag1997}
R. Jagannathan, \emph{$(P,Q)$-special functions}, in: Special functions and differential equations. Proceedings of a workshop, WSSF '97, Madras, India, January 13--24, 1997, 158--164.

\bibitem{KK1992}
J. Katriel and M. Kibler, \emph{Normal ordering for deformed boson operators and operator-valued deformed Stirling numbers}, J. Phys. A: Math. Gen. \textbf{25} (1992), 2683--2691

\bibitem{EKLS1993}
E.E. Kirkman and L.W. Small, \emph{$q$-analogs of harmonic oscillators and related rings}, Israel J. Math. \textbf{81} (1993), 111--127.

\bibitem{KHT1995}
H.T. Koelink, \emph{Addition formulas for $q$-special functions}. In: Proceedings Special Functions, $q$-Series and Related Topics, Toronto, June (1995).

\bibitem{KTH1997}
T.H. Koornwinder, \emph{Special functions and $q$-commuting variables}, Fields Inst. Commun. \textbf{14} (1997), 131--166.

\bibitem{YIM1988} 
Y.I. Manin, {\it Quantum Groups and Noncommutive Geometry}, Centre de Recherches Mathématiques, Montréal (1988).

\bibitem{TLM2} 
T. Mansour, L. Oussi and M. Schork, \emph{Normal ordering in the $(p,q)$-deformed generalized Weyl algebra. II: Interpretation in terms of rook placements}, Available at: \url{https://doi.org/10.48550/arXiv.2607.01141}.

\bibitem{TLM3} 
T. Mansour, L. Oussi and M. Schork, \emph{Normal ordering in the $(p,q)$-deformed generalized Weyl algebra. III: The binomial formula}, Ramanujan J. \textbf{71} (2026), Article: 38.

\bibitem{TMMS2011} 
T. Mansour and M. Schork, \emph{The commutation relation $xy=qyx+hf(y)$ and Newton's binomial formula}, Ramanujan J. \textbf{25} (2011), 405--445.

\bibitem{TMMS2016} T. Mansour, M. Schork, {\em Commutation relations, normal ordering, and Stirling numbers}, CRC Press, Boca Raton, FL (2016).

\bibitem{GCMDCH} 
J.C. Mason and D.C. Handscomb, \emph{Chebyshev Polynomials}, Chapman and Hall/CRC, 2003.

\bibitem{OT2024}
A. Ouahhabi and E.H. Tahri, \emph{On generalized normal ordering and generalized Stirling operators}, Phys. Lett. A {\bf 525} (2024), Article 129917.

\bibitem{OT2026}
A. Ouahhabi and E.H. Tahri, \emph{On generalized Bell numbers and normal ordering problem}, Int. J. Geom. Methods Mod. Phys. {\bf 23} (2026), Article 2550285.

\bibitem{LO2020}
L. Oussi, \emph{A $(p, q)$-deformed recurrence for the Bell numbers}, J. Integer Seq. {\bf 23} (2020), Article 20.5.2.

\bibitem{LO2021}
L. Oussi, \emph{$(p, q)$-analogues of the generalized Touchard polynomials and Stirling numbers}, Indag. Math. \textbf{33} (2022), 664--681.

\bibitem{LO2024}
L. Oussi, \emph{A note on the $(p, q)$-derivative operator}, Int. J. Appl. Comput. Math. \textbf{10} (2024), Article 172.

\bibitem{HSAP1950}
H.S.A. Potter, \emph{On the latent roots of quasi-commutative matrices}, Amer. Math. Monthly \textbf{57} (1950), 321--322.

\bibitem{RAS1958}
R.A. Sack, \emph{Taylor’s theorem for shift operators}, Philos. Mag. VIII (1958), 497--503.

\bibitem{MS2021}
M. Schork, {\em Recent developments in combinatorial aspects of normal ordering}, Enumer. Combin. Appl. {\bf 1} (2021), Article S2S2. 

\bibitem{MPS1953}
M.P. Sch\"{u}tzenberger, \emph{Une interpr\'{e}tation de certains solutions de l'\'{e}quation fonctionnelle: $F(x+y) = F(x)F(y)$},  C. R. Acad. Sci. Paris \textbf{236} (1953), 352--353.

\bibitem{AV2005}
A. Varvak, \emph{Rook numbers and the normal ordering problem}, J. Combin. Theory Ser. A. \textbf{112} (2005), 292--307.

\bibitem{MWDW1991}
M. Wachs and D. White, \emph{$p,q$-Stirling numbers and set partition statistics}, J. Combin. Theory Ser. A. \textbf{56} (1991), 27--46.

\bibitem{RMW1967}
R.M. Wilcox, \emph{Exponential operators and parameter differentiation in quantum physics}, J. Math. Phys. \textbf{8} (1967), 962–-982.

\bibitem{WW1974}
W. Witschel, \emph{Ordered operator expansions by comparison}, J. Phys. A. \textbf{8} (1974), 143--154.

\end{thebibliography}
\end{document}